\documentclass{article}
\usepackage[T1]{fontenc}
\usepackage[latin9]{inputenc}
\usepackage{geometry}
\usepackage{color}
\usepackage{mathrsfs}
\usepackage{url}
\usepackage{amsthm}
\usepackage{amsmath}
\usepackage{amssymb}
\usepackage{esint}
\usepackage[unicode=true,
 bookmarks=true,bookmarksnumbered=false,bookmarksopen=false,
 breaklinks=false,pdfborder={0 0 0},backref=page,colorlinks=true]
 {hyperref}
\hypersetup{pdftitle={Spectral Risk Measures},
 pdfauthor={Alois Pichler}}

\makeatletter

\newcommand{\noun}[1]{\textsc{#1}}

\usepackage{enumitem}		
\theoremstyle{plain}
\newtheorem{thm}{\protect\theoremname}
  \theoremstyle{plain}
  \newtheorem{prop}[thm]{\protect\propositionname}
  \theoremstyle{definition}
  \newtheorem{defn}[thm]{\protect\definitionname}
  \theoremstyle{remark}
  \newtheorem{rem}[thm]{\protect\remarkname}
  \theoremstyle{definition}
  \newtheorem{example}[thm]{\protect\examplename}
  \theoremstyle{plain}
  \newtheorem{lem}[thm]{\protect\lemmaname}

\usepackage{dsfont}   

\usepackage{ifpdf}
\ifpdf
          \usepackage{pdfsync} 	
          \IfFileExists{lmodern.sty}{\usepackage{lmodern}}{}	
\else		
          \usepackage[active]{srcltx}
\fi

\newcommand{\one}{\mathds 1}  
\newcommand{\AVaR}{{\sf AV@R}}
\newcommand{\VaR}{{\sf V@R}}
\newcommand{\sign}{{\operatorname{sign}}}

\newcommand{\id}{{\operatorname{id}}}

\newcommand{\essinf}{\operatornamewithlimits{ess\,inf}}
\newcommand{\esssup}{\operatornamewithlimits{ess\,sup}}

\makeatother

  \providecommand{\definitionname}{Definition}
  \providecommand{\examplename}{Example}
  \providecommand{\lemmaname}{Lemma}
  \providecommand{\propositionname}{Proposition}
  \providecommand{\remarkname}{Remark}
\providecommand{\theoremname}{Theorem}

\begin{document}

\title{The Natural Banach Space for Version Independent Risk Measures}

\author{Alois Pichler~%
\thanks{Norwegian University of Science and Technology.\protect \\
Contact: \protect\url{alois.pichler@iot.ntnu.no}%
}}
\maketitle
\begin{abstract}
Risk measures, or coherent measures of risk are often considered on
the space $L^{\infty}$, and important theorems on risk measures build
on that space. Other risk measures, among them the most important
risk measure---the Average Value-at-Risk---are well defined on the
larger space $L^{1}$ and this seems to be the natural domain space
for this risk measure. Spectral risk measures constitute a further
class of risk measures of central importance, and they are often considered
on some $L^{p}$ space. But in many situations this is possibly unnatural,
because any $L^{p}$ with $p>p_{0}$, say, is suitable to define the
spectral risk measure as well. In addition to that risk measures have
also been considered on Orlicz and Zygmund spaces. So it remains for
discussion and clarification, what the natural domain to consider
a risk measure is? 

This paper introduces a norm, which is built from the risk measure,
and a Banach space, which carries the risk measure in a natural way.
It is often strictly larger than its original domain, and obeys the
key property that the risk measure is finite valued and continuous
on that space in an elementary and natural way. 

\textbf{Keywords:} Risk Measures, Rearrangement Inequalities, Stochastic
Dominance, Dual Representation\\
\textbf{Classification:} 90C15, 60B05, 62P05
\end{abstract}

\section{Introduction}

This paper addresses coherent measures of risk (risk measures, for
short) and the natural domain (the natural space), where they can
be considered. Coherent measures of risk have been introduced in the
seminal paper \cite{Artzner1999} in an axiomatic way and have been
investigated in a series of subsequent papers in mathematical finance
since then. In the actuarial literature, however, risk measures and
axiomatic treatments have been considered already earlier, for example
in Denneberg (\cite{Denneberg1989}) and in this journal by Wang et
al. (\cite{Wang1997}). 

We state the axioms (cf.~\cite{Artzner1997}) for a convex risk measure
$\rho$, mapping $\mathbb{R}$\nobreakdash-valued random variables
into the real numbers $\mathbb{R}$ or to $+\infty$. Here, the initial
axioms have been adapted to follow the interpretation of loss instead
of profit---the common modification in insurance---in the usual and
appropriate way.\pagebreak[2]
\begin{enumerate}[label=(M)]
\item \emph{\noun{Monotonicity: \label{enu:Monotonicity} }}$\rho\left(Y_{1}\right)\le\rho\left(Y_{2}\right)$
whenever $Y_{1}\le Y_{2}$ almost surely;
\end{enumerate}

\begin{enumerate}[label=(H)]
\item \emph{\noun{Positive~homogeneity: \label{enu:Homogeneity} }}$\rho\left(\lambda Y\right)=\lambda\rho\left(Y\right)$
whenever $\lambda>0$;
\end{enumerate}

\begin{enumerate}[label=(C)]
\item \emph{\noun{Convexity: \label{enu:Convexity} }}$\rho\left(\left(1-\lambda\right)Y_{0}+\lambda Y_{1}\right)\le\left(1-\lambda\right)\rho\left(Y_{0}\right)+\lambda\rho\left(Y_{1}\right)$
for $0\le\lambda\le1$;
\end{enumerate}

\begin{enumerate}[label=(T)]
\item \emph{\noun{Translation~equivariance}}%
\footnote{In an economic or monetary environment this is often called \emph{\noun{Cash
invariance}} instead.%
}\emph{\noun{: \label{enu:Equivariance} }}$\rho\left(Y+c\right)=\rho\left(Y\right)+c$
if $c\in\mathbb{R}$.
\end{enumerate}
The main observation in this paper starts with the fact that the risk
measure $\rho$ can be associated in a natural way with a seminorm,
which is a norm in important cases. It is an elementary property that
the risk measure is continuous with respect to the norm introduced.

We investigate this new norm for specific risk measures, starting
with spectral risk measures. It turns out that the domain, where the
spectral risk measure can be defined in a meaningful way, is always
strictly larger than $L^{\infty}$. The respective space is a Banach
space, and we study its topology, which can be compared with $L^{p}$
spaces. However, the topology always differs from the topology of
an $L^{p}$ space (cf. \cite{Filipovic2012}).

A risk measure $\rho$---being a convex function---has a convex conjugate
function, and the Fenchel--Moreau theorem allows recovering the initial
function, the initial risk measure $\rho$ in our situation. The convex
conjugate function involves the dual of the initial space, for this
reason it is essential to understand the dual of the Banach space
associated with the risk measure. The norm on the dual space measures
the growth of the random variable by involving second order stochastic
dominance relations.

\medskip{}

It is elaborated moreover in this paper that a risk measure cannot
be defined in a meaningful way on a space larger than $L^{1}$.

\subsection*{The domain and the co-domain of spectral risk measures}

The axioms characterizing risk measures have been stated above without
giving the domain and the co-domain precisely. Indeed, important results
are well known when considering $\rho$ as a function on $L^{\infty}$,
$\rho:\, L^{\infty}\to\mathbb{R}$: the results include Kusuoka's
representation (cf.~\cite{Kusuoka} and \eqref{eq:Kusuoka} below)
and results on continuity. We state the following example.
\begin{prop}
\label{prop:LInfty}Every $\mathbb{R}$\nobreakdash-valued risk measure
$\rho$ on $L^{\infty}$ is Lipschitz-continuous with constant $1$,
it satisfies $\left|\rho\left(Y_{2}\right)-\rho\left(Y_{1}\right)\right|\le\left\Vert Y_{2}-Y_{1}\right\Vert _{\infty}$. \end{prop}
\begin{proof}
See, e.g., \cite[Lemma 4.3]{Follmer2004} for a proof.
\end{proof}
In many situations, for example when considering the trivial risk
measure $\rho\left(\cdot\right):=\mathbb{E}\left(\cdot\right)$ or
the Average Value-at-Risk, the domain $L^{\infty}$ is not satisfactory
large enough, the domain $L^{1}$ is perhaps more natural and convenient
to consider in this situation.

Depending on the domain chosen for a risk measure, the co-domain is
often specified to be $\mathbb{R}$, or the extended reals $\mathbb{R}\cup\left\{ \infty\right\} $,
in some publications even $\mathbb{R}\cup\left\{ \infty,\,-\infty\right\} $.
In this context it should be emphasized that there is an intimate
relationship between the properties \emph{continuity} of a risk measure
and its \emph{range}, the following important result clarifies the
connections: 
\begin{prop}
Consider a $\mathbb{R}\cup\left\{ \infty\right\} $\nobreakdash-valued,
lsc. risk measure $\rho$ defined on $L^{p}$, $1\le p<\infty$, satisfying
\ref{enu:Monotonicity}, \ref{enu:Convexity} and \ref{enu:Equivariance}.
Suppose further that $\left\{ \rho<\infty\right\} $ has a nonempty
interior. Then $\rho$ is finite valued and continuous on the entire
$L^{p}$.
\end{prop}
The proof is contained in \cite{Ruszczynski2006} and in \cite{RuszczynskiShapiro2009},
Proposition 6.7. The preceding discussion of the latter reference
also contains the following reformulation of the statement, which
is more striking: A risk measure satisfying \ref{enu:Monotonicity},
\ref{enu:Convexity} and \ref{enu:Equivariance} is either finite
valued and continuous on the \emph{entire} $L^{p}$, or it takes the
value $+\infty$ on a dense subset.

\medskip{}

Both results suggest to consider $\mathbb{R}$ (i.e. $\mathbb{R}\backslash\left\{ \pm\infty\right\} $)
valued risk measures solely, because these are precisely the finite
valued and continuous risk measures.

\medskip{}

Outline of the paper: The following Section~\ref{sec:Norm} introduces
the associated norm and elaborates its elementary property. The subsequent
section, Section~\ref{sec:LSigma}, addresses an elementary risk
measure, the spectral risk measure. This risk measure is elementary,
as every version independent risk measure can be built from spectral
risk measures. 

A space is introduced, which we call the space of \emph{natural domain},
which is as large as possible to carry a spectral risk measure. It
is verified that the associated space is a Banach space. The new norm
can be used in a natural way to extend the domain of elementary risk
measures, and it is elaborated which $L^{p}$ spaces the space of
natural domain comprises. 

This section contains moreover the remarkable result, that there is
no finite valued risk measure on a space larger than $L^{1}$.

We study further the topological dual of the Banach space introduced
(Section~\ref{sec:LDual}). It turns out the dual norm can be characterized
by use of the Average Value-at-Risk, the simplest risk measure, and
by second order stochastic dominance. The investigations are pushed
further to more general risk measures, and an even more general Banach
space to carry a general risk measure is highlighted in Section~\ref{sec:LGeneral}.

\section{\label{sec:Norm}The norm associated with a risk measure}

The results presented in this paper start along with the observation
that a risk measure $\rho$ induces a (semi-)norm in the following
elementary way. 
\begin{defn}
Let $L$ be a vector space of $\mathbb{R}$\nobreakdash-valued random
variables on $\left(\Omega,\mathcal{F},P\right)$ and $\rho:\, L\to\mathbb{R}\cup\left\{ -\infty,\infty\right\} $
be a risk measure. Then 
\[
\left\Vert \cdot\right\Vert _{\rho}:=\rho\left(\left|\cdot\right|\right)
\]
is called \emph{associated norm}, associated with the risk measure
$\rho$. \end{defn}
\begin{rem}
If no confusion may occur we shall simply write $\left\Vert \cdot\right\Vert $
to refer to $\left\Vert \cdot\right\Vert _{\rho}$.
\end{rem}
The following proposition verifies that $\left\Vert \cdot\right\Vert _{\rho}$
is indeed a seminorm on the appropriate vector space. 
\begin{prop}[Finiteness, and the seminorm property]
\label{prop:Finiteness}Let $\rho$ be a risk measure on a vector
space of $\mathbb{R}$\nobreakdash-valued random variables. Then
$\left\Vert \cdot\right\Vert =\rho\left(\left|\cdot\right|\right)$
is a seminorm on $L:=\left\{ Y:\,\rho\left(\left|Y\right|\right)<\infty\right\} $
and $\rho$ is finite valued on $L$.\end{prop}
\begin{proof}
We show first that that $\rho$ is $\mathbb{R}$\nobreakdash-valued
on $L=\left\{ Y:\,\rho\left(\left|Y\right|\right)<\infty\right\} $.
For this observe that $Y\le\left|Y\right|$, and by monotonicity thus
$\rho\left(Y\right)\le\rho\left(\left|Y\right|\right)=\left\Vert Y\right\Vert $.
Moreover it holds that $\rho\left(0\right)=0$ %
\footnote{Otherwise, $\rho\left(0\right)=\rho\left(2\cdot0\right)=2\cdot\rho\left(0\right)$
would imply $1=2$, a contradiction.%
} and thus 
\[
0=2\cdot\rho\left(\frac{1}{2}Y+\frac{1}{2}\left(-Y\right)\right)\le2\cdot\left(\frac{1}{2}\rho\left(Y\right)+\frac{1}{2}\rho\left(-Y\right)\right)=\rho\left(Y\right)+\rho\left(-Y\right),
\]
such that $-\rho\left(Y\right)\le\rho\left(-Y\right)$. Now $-Y\le\left|Y\right|$
and, again by monotonicity, $-\rho\left(Y\right)\le\rho\left(-Y\right)\le\rho\left(\left|Y\right|\right)=\left\Vert Y\right\Vert $.
Summarizing thus $\left|\rho\left(Y\right)\right|\le\left\Vert Y\right\Vert $,
such that $\rho$ is finite valued on $L$. 

Note that 
\[
\left\Vert \lambda\cdot Y\right\Vert =\rho\left(\left|\lambda\cdot Y\right|\right)=\rho\left(\left|\lambda\right|\cdot\left|Y\right|\right)=\left|\lambda\right|\cdot\rho\left(\left|Y\right|\right)=\left|\lambda\right|\cdot\left\Vert Y\right\Vert ,
\]
 and $\left\Vert \cdot\right\Vert $ thus is positively homogeneous.

Next it follows from monotonicity, positive homogeneity and convexity
that 
\begin{align*}
\left\Vert Y_{1}+Y_{2}\right\Vert = & \rho\left(\left|Y_{1}+Y_{2}\right|\right)\le\rho\left(\left|Y_{1}\right|+\left|Y_{2}\right|\right)=2\cdot\rho\left(\frac{1}{2}\left|Y_{1}\right|+\frac{1}{2}\left|Y_{2}\right|\right)\\
 & \le2\cdot\left(\frac{1}{2}\rho\left(\left|Y_{1}\right|\right)+\frac{1}{2}\rho\left(\left|Y_{2}\right|\right)\right)=\rho\left(\left|Y_{1}\right|\right)+\rho\left(\left|Y_{2}\right|\right)\\
 & =\left\Vert Y_{1}\right\Vert +\left\Vert Y_{2}\right\Vert ,
\end{align*}
and this is the triangle inequality. 
\end{proof}
The next proposition elaborates, that the risk measure is continuous
with respect to its associated norm. This consistency result on continuity
generalizes Proposition~\ref{prop:LInfty}.
\begin{prop}[Continuity]
\label{prop:Continuity}Let $\rho$ be a risk measure, defined on
a vector space of $\mathbb{R}$\nobreakdash-valued random variables.
Then $\rho$ is Lipschitz continuous with constant $1$ with respect
to the seminorm $\left\Vert \cdot\right\Vert =\rho\left(\left|\cdot\right|\right)$.\end{prop}
\begin{proof}
As for continuity note that 
\begin{eqnarray*}
\rho\left(Y_{2}\right) & = & 2\cdot\rho\left(\frac{1}{2}Y_{1}+\frac{1}{2}\left(Y_{2}-Y_{1}\right)\right)\\
 & \le & 2\left(\frac{1}{2}\rho\left(Y_{1}\right)+\frac{1}{2}\rho\left(Y_{2}-Y_{1}\right)\right)\le\rho\left(Y_{1}\right)+\rho\left(\left|Y_{2}-Y_{1}\right|\right)
\end{eqnarray*}
by convexity and monotonicity. It follows that $\rho\left(Y_{2}\right)-\rho\left(Y_{1}\right)\le\left\Vert Y_{2}-Y_{1}\right\Vert .$
Interchanging the roles of $Y_{1}$ and $Y_{2}$ reveals that 
\[
\left|\rho\left(Y_{2}\right)-\rho\left(Y_{1}\right)\right|\le\left\Vert Y_{2}-Y_{1}\right\Vert ,
\]
the assertion. To accept that the Lipschitz constant $1$ cannot be
improved consider the particular choices $Y_{1}:=0$ and $Y_{2}:=\one$
in view of translation equivariance \ref{enu:Equivariance}. 
\end{proof}

\section{\label{sec:LSigma}Spectral risk measures}

Among the initial attempts to introduce premium principles to price
insurance contracts are distorted probabilities, a concept which can
be summarized nowadays by distorted acceptability functionals (cf.~\cite{PflugRomisch2007})
or spectral risk measures. Spectral risk measures---or the  weighted
Value-at-Risk (cf.~\cite{Cherny2006}), which is a more suggestive
term---have been considered for example in \cite{Acerbi2002,Acerbi2002a}.
This risk measure involves the Value-at-Risk at level $p$, 
\[
\VaR_{p}\left(Y\right):=F_{Y}^{-1}\left(p\right):=\inf\left\{ y:\, P\left(Y\le y\right)\ge p\right\} ,
\]
which is the left-continuous, lower semi-continuous (lsc.) \emph{quantile};
the spectral risk measure (or weighted $\VaR$) then is the functional
\begin{equation}
\rho_{\sigma}\left(Y\right):=\int_{0}^{1}\sigma\left(u\right)\VaR_{u}\left(Y\right)\,\mathrm{d}u,\label{eq:SRM}
\end{equation}
mapping a random variable $Y$ to a real number, if the integral exists. 

The function $\sigma:\,\left[0,1\right]\to\mathbb{R}_{0}^{+}$, called
the \emph{spectrum} or \emph{spectral function}, is a weight function.
To build a reasonable premium principle the function $\sigma$ should
obey some properties to be consistent with the axioms imposed on risk
measures: first, associating $Y$ with loss, $\sigma$ should evaluate
to nonnegative reals, $\mathbb{R}_{0}^{+}$. Higher losses should
be weighted higher, thus $\sigma$ should be nondecreasing. And finally,
as $\sigma$ represents a weight function, it is natural to request
$\int_{0}^{1}\sigma\left(u\right)\,\mathrm{d}u=1$.

An important, elementary spectral  risk measure satisfying all axioms
above is the Average Value-at-Risk, which is specified by the spectral
function 
\[
\sigma_{\alpha}\left(u\right):=\begin{cases}
0 & \text{if }u<\alpha\\
\frac{1}{1-\alpha} & \text{else},
\end{cases}
\]
that is

\begin{equation}
\AVaR_{\alpha}\left(Y\right):=\frac{1}{1-\alpha}\int_{\alpha}^{1}\VaR_{u}\left(Y\right)\mathrm{d}u\qquad\left(\alpha<1\right),\label{eq:AVaR}
\end{equation}
and for $\alpha=1$ the Average Value-at-Risk per definition is 
\[
\AVaR_{1}\left(Y\right):=\lim_{\alpha\nearrow1}\AVaR_{\alpha}\left(Y\right)=\esssup Y\qquad\left(\alpha=1\right).
\]

\subsection*{The domain of spectral risk measures}

It is obvious that the Average Value-at-Risk ($\alpha<1$) may be
well defined on $L^{1}$, with the result that 
\[
\left|\AVaR_{\alpha}\left(Y\right)\right|\le\frac{1}{1-\alpha}\mathbb{E}\left|Y\right|=\frac{1}{1-\alpha}\left\Vert Y\right\Vert _{1}<\infty\qquad\left(Y\in L^{1}\right),
\]
that means that $\AVaR_{\alpha}$ is finite valued whenever $Y\in L^{1}$.
This is not the case, however, for $\alpha=1$: a restriction to the
smaller space $L^{\infty}\subset L^{1}$ is necessary in order to
ensure that $\AVaR_{1}$ is finite valued, 
\[
\left|\AVaR_{1}\left(Y\right)\right|\le\left\Vert Y\right\Vert _{\infty}<\infty\qquad\left(Y\in L^{\infty}\right).
\]

Even more peculiarities appear when considering the spectral function
$\sigma\left(u\right):=\frac{1}{2\sqrt{1-u}}$. Clearly, $\sigma\in L^{q}$
whenever $q<2$, but $\sigma\notin L^{2}$. Hölder's inequality can
be employed to insure that $\rho_{\sigma}$ is finite valued on $L^{p}$
($p>2$, $\frac{1}{q}+\frac{1}{p}=1$), because 
\[
\left|\rho_{\sigma}\left(Y\right)\right|\le\left\Vert \sigma\right\Vert _{q}\cdot\left(\int_{0}^{1}F_{Y}^{-1}(u)^{p}\right)^{\frac{1}{p}}=\frac{1}{2}\left(\frac{2}{2-q}\right)^{\frac{1}{q}}\cdot\left\Vert Y\right\Vert _{p},
\]
and the constant $\frac{1}{2}\left(\frac{2}{2-q}\right)^{\frac{1}{q}}$
again exceeds every finite bound whenever $q$ approaches $2$ from
below. 

So what is a good space to consider $\rho_{\sigma}$? Any $L^{p}$
($p>2$) guarantees that $\rho_{\sigma}$ is finite valued and continuous,
but $L^{2}$ is obviously too large. The naïve choice $\bigcup_{p>2}L^{p}$
does not have a satisfying norm, or topology neither. (See, for different
configurations, \cite{Cheridito2008,Cheridito2009a}.)

\medskip{}

\subsection*{Further properties and importance of spectral risk measures}

A well known and essential representation of risk measures was elaborated
by Kusuoka in \cite{Kusuoka} (see \cite{Schachermayer2006} for the
statement presented below). Kusuoka's result considers risk measures
on $L^{\infty}$ which are \emph{version independent} (also: \emph{law
invariant}), i.e. which satisfy $\rho\left(Y\right)=\rho\left(Y^{\prime}\right)$
whenever $Y$ and $Y^{\prime}$ share the same law, that is if $P\left(Y\le y\right)=P\left(Y^{\prime}\le y\right)$
for every $y\in\mathbb{R}$.
\begin{thm}[Kusuoka's representation]
\label{thm:Kusuoka}A version independent risk measure $\rho$ on
$L^{\infty}$ of an atom-less probability space $\left(\Omega,\,\mathcal{F},\, P\right)$
has the representation 
\begin{equation}
\rho\left(Y\right)=\sup_{\mu\in\mathscr{M}}\int_{0}^{1}\AVaR_{\alpha}\left(Y\right)\,\mu\left(\mathrm{d}\alpha\right),\label{eq:Kusuoka}
\end{equation}
where $\mathscr{M}$ is a collection of probability measures on $\left[0,\,1\right]$.
\end{thm}

\paragraph{Kusuoka representation \emph{of} a spectral risk measure. }

The Kusuoka representation of a spectral risk measure $\rho_{\sigma}$
is provided by the probability measure $\mu_{\sigma}\left(\left(a,b\right]\right):=\int_{a}^{b}\mathrm{d}\mu_{\sigma}\left(\alpha\right)$
on $\left[0,\,1\right]$, where $\mu_{\sigma}$ is the nondecreasing
function 
\begin{equation}
\mu_{\sigma}\left(p\right):=\left(1-p\right)\sigma\left(p\right)+\int_{0}^{p}\sigma\left(u\right)\mathrm{d}u\quad(0\le p\le1),\qquad\mu_{\sigma}\left(p\right):=0\quad\left(p<0\right),\label{eq:MuSigma}
\end{equation}
which satisfies $\mu_{\sigma}\left(1\right)=1$ and $\mathrm{d}\mu_{\sigma}\left(p\right)=\left(1-p\right)\mathrm{d}\sigma\left(p\right)$.
It holds that 
\begin{equation}
\rho_{\sigma}\left(Y\right)=\int_{0}^{1}\AVaR_{\alpha}\left(Y\right)\,\mu_{\sigma}\left(\mathrm{d}\alpha\right),\label{eq:5}
\end{equation}
which exposes the Kusuoka representation of a spectral risk measure
(cf.~\cite{ShapiroAlois}).

\paragraph{Kusuoka representation \emph{by} spectral risk measures. }

Conversely, any measure $\mu$ (provided that $\mu\left(\left\{ 1\right\} \right)=0$)
of the representation \eqref{eq:Kusuoka} can be related to the function
\begin{equation}
\sigma_{\mu}\left(\alpha\right)=\int_{0}^{\alpha}\frac{1}{1-u}\,\mu(\mathrm{d}u),\label{eq:Sigma}
\end{equation}
and it holds that 
\[
\int_{0}^{1}\AVaR_{\alpha}\left(Y\right)\,\mu\left(\mathrm{d}\alpha\right)=\int_{0}^{1}\sigma_{\mu}\left(\alpha\right)\VaR_{\alpha}\left(Y\right)\,\mathrm{d}\alpha=\rho_{\sigma_{\mu}}\left(Y\right),
\]
which is a spectral risk measure. 

But even the requirement $\mu\left(\left\{ 1\right\} \right)=0$ can
be dropped: indeed, there is a set $\mathscr{S}$ of continuous (and
thus bounded) spectral functions on $\left[0,\,1\right]$, such that
the relation 
\begin{equation}
\rho\left(Y\right)=\sup_{\mu\in\mathscr{M}}\int_{0}^{1}\AVaR_{\alpha}\left(Y\right)\,\mu\left(\mathrm{d}\alpha\right)=\sup_{\sigma\in\mathscr{S}}\int_{0}^{1}\VaR_{\alpha}\left(Y\right)\sigma\left(\alpha\right)\,\mathrm{d}\alpha=\sup_{\sigma\in\mathscr{S}}\rho_{\sigma}\left(Y\right)\label{eq:Spectral}
\end{equation}
holds (cf.~\cite{PflugPichler}). This again exposes the importance
of spectral risk measures, as every version independent risk measure
$\rho$ can be built from spectral risk measures by \eqref{eq:Spectral}. 

\medskip{}

Recall that Kusuoka's representation builds on the space $L^{\infty}$.
But again it is not clear, if, and to which larger space this risk
measure can be extended, because every $\sigma$ might allow a different
domain.

\section{The space of natural domain,~$L_{\sigma}$}

Let $\sigma$ be a nonnegative, nondecreasing, integrable function
with $\int_{0}^{1}\sigma(u)\mathrm{d}u=1$. For $Y$ a random variable
we consider the function 
\[
\rho_{\sigma}\left(Y\right)=\int_{0}^{1}\sigma\left(u\right)\, F_{Y}^{-1}\left(u\right)\,\mathrm{d}u
\]
already defined in \eqref{eq:SRM}. For $\sigma\in L^{1}$ (which
is a minimal requirement to insure that $\int_{0}^{1}\sigma(u)\mathrm{d}u=1$),
$\rho_{\sigma}$ is certainly well defined for $Y\in L^{\infty}$,
but for other random variables the integral possibly diverges. And
it might diverge to $+\infty$, to $-\infty$, or be even of the indefinite
form $\infty-\infty$. The following definition respects the finiteness
of the spectral risk measure in view of Proposition~\ref{prop:Finiteness}.
\begin{defn}
The \emph{natural domain} corresponding to a spectral risk measure
$\rho_{\sigma}$ induced by a spectral function $\sigma$ is 
\[
L_{\sigma}:=\left\{ Y\in L^{0}:\,\left\Vert Y\right\Vert _{\sigma}<\infty\right\} ,
\]
where 
\[
\left\Vert Y\right\Vert _{\sigma}:=\rho_{\sigma}\left(\left|Y\right|\right).
\]

\end{defn}
Note that $\left|Y\right|\ge0$ is positive, such that $F_{\left|Y\right|}^{-1}\left(\cdot\right)\ge0$
is positive as well and the condition $\rho_{\sigma}\left(\left|Y\right|\right)<\infty$
makes perfect sense for any measurable random variable $Y\in L^{0}$.
\begin{rem}
The seminorm $\left\Vert \cdot\right\Vert _{\sigma}$ has the representation
\[
\left\Vert Y\right\Vert _{\sigma}=\int_{0}^{\infty}\tau_{\sigma}\left(F_{\left|Y\right|}(y)\right)\mathrm{d}y
\]
in terms of the cdf. $F_{\left|Y\right|}$ directly, without involving
the inverse $F_{\left|Y\right|}^{-1}$ ($\tau_{\sigma}\left(\alpha\right):=\int_{\alpha}^{1}\sigma(u)\mathrm{d}u$).\end{rem}
\begin{prop}
$\left\Vert \cdot\right\Vert _{\sigma}=\rho_{\sigma}\left(\left|\cdot\right|\right)$
is a norm on $L_{\sigma}$.\end{prop}
\begin{proof}
It was already shown in Proposition~\ref{prop:Finiteness} that $\left\Vert \cdot\right\Vert _{\sigma}$
is a seminorm. What remains to be shown is that $\left\Vert \cdot\right\Vert _{\sigma}$
separates points. For this recall that $\sigma$ is positive, nondecreasing,
and satisfies $\int_{0}^{1}\sigma\left(p\right)\mathrm{d}p=1$, and
$F_{\left|Y\right|}\left(\cdot\right)$ is a nondecreasing and positive
function as well. Hence if $\int_{0}^{1}\sigma\left(p\right)F_{\left|Y\right|}^{-1}\left(p\right)\mathrm{d}p=0$,
then $F_{\left|Y\right|}^{-1}\left(\cdot\right)\equiv0$, that is
$Y=0$ almost everywhere. The function $\left\Vert \cdot\right\Vert _{\sigma}$
thus separates points in $L_{\sigma}$ and $\left\Vert \cdot\right\Vert _{\sigma}$
hence is a norm.
\end{proof}
The next theorem already elaborates that the set $L_{\sigma}$ is
large enough and at least contains $L^{p}$, whenever $\sigma\in L^{q}$
(and the exponents are conjugate, $\frac{1}{p}+\frac{1}{q}=1$). 
\begin{thm}[Comparison with $L^{p}$]
 \label{thm:Inclusions}Let $\sigma$ be fixed. 
\begin{enumerate}[label=(\roman{enumi})]
\item \label{enu:Lp}If $\sigma\in L^{q}$ for some $q\in\left[1,\infty\right]$
with conjugate exponent $p$, then 
\[
L^{\infty}\subset L^{p}\subset L_{\sigma}\subset L^{1}
\]
and 
\begin{equation}
\left\Vert Y\right\Vert _{1}\le\left\Vert Y\right\Vert _{\sigma}\le\left\Vert \sigma\right\Vert _{q}\cdot\left\Vert Y\right\Vert _{p}\label{eq:sigmaLp}
\end{equation}
whenever $Y\in L^{p}$.
\item \label{enu:LInfty}For $\sigma$ bounded (i.e. $\sigma\in L^{\infty}$)
it holds moreover that $L_{\sigma}=L^{1}$, the norms are equivalent
and satisfy 
\[
\left\Vert Y\right\Vert _{1}\le\left\Vert Y\right\Vert _{\sigma}\le\left\Vert \sigma\right\Vert _{\infty}\cdot\left\Vert Y\right\Vert _{1}.
\]

\end{enumerate}
\end{thm}
It follows in particular from \ref{enu:LInfty} that $P\left(A\right)\le\left\Vert \one_{A}\right\Vert _{\sigma}\le1$
for measurable sets $A$, and $\left\Vert Y\right\Vert _{\sigma}=\left\Vert Y\right\Vert _{1}$
for the function being constantly $1$ ($\sigma=\one$).
\begin{proof}
Note that $\int_{0}^{1}\sigma\left(u\right)\mathrm{d}u=1$ and $\sigma\left(\cdot\right)$
is nondecreasing, hence there is a $\tilde{u}\in\left(0,\,1\right)$
such that $\sigma\left(u\right)\le1$ for $u<\tilde{u}$ and $\sigma\left(u\right)\ge1$
for $u>\tilde{u}$. Note as well that $\int_{0}^{\tilde{u}}1-\sigma\left(u\right)\,\mathrm{d}u=\int_{\tilde{u}}^{1}\sigma\left(u\right)-1\,\mathrm{d}u$.
Then it follows that 
\begin{align*}
\int_{0}^{\tilde{u}}\left(1-\sigma\left(u\right)\right)F_{\left|Y\right|}^{-1}\left(u\right)\mathrm{d}u & \le\int_{0}^{\tilde{u}}\left(1-\sigma\left(u\right)\right)F_{\left|Y\right|}^{-1}\left(\tilde{u}\right)\mathrm{d}u\\
 & =\int_{\tilde{u}}^{1}\left(\sigma\left(u\right)-1\right)F_{\left|Y\right|}^{-1}\left(\tilde{u}\right)\mathrm{d}u\le\int_{\tilde{u}}^{1}\left(\sigma\left(u\right)-1\right)F_{\left|Y\right|}^{-1}\left(u\right)\mathrm{d}u,
\end{align*}
because $F_{\left|Y\right|}^{-1}\left(\cdot\right)$ is increasing.
After rearranging thus 
\[
\left\Vert Y\right\Vert _{1}=\mathbb{E}\left|Y\right|=\int_{0}^{1}F_{\left|Y\right|}^{-1}\left(u\right)\mathrm{d}u\le\int_{0}^{1}F_{\left|Y\right|}^{-1}\left(u\right)\sigma\left(u\right)\mathrm{d}u=\rho_{\sigma}\left(\left|Y\right|\right)=\left\Vert Y\right\Vert _{\sigma},
\]
which is the first assertion. The inclusion $L_{\sigma}\subset L^{1}$
is immediate as well, as $\left\Vert Y\right\Vert _{\sigma}<\infty$
implies that $\left\Vert Y\right\Vert _{1}<\infty$. 

The remaining inequality 
\[
\left\Vert Y\right\Vert _{\sigma}=\int_{0}^{1}F_{\left|Y\right|}^{-1}\left(u\right)\sigma\left(u\right)\mathrm{d}u\le\left(\int_{0}^{1}\sigma\left(u\right)^{q}\right)^{\frac{1}{q}}\cdot\left(\int_{0}^{1}F_{\left|Y\right|}^{-1}\left(u\right)^{p}\right)^{\frac{1}{p}}=\left\Vert \sigma\right\Vert _{q}\cdot\left(\mathbb{E}\left|Y\right|^{p}\right)^{\frac{1}{p}}
\]
is Hölder's inequality.\end{proof}
\begin{rem}
The inequality $\left\Vert Y\right\Vert _{1}\le\left\Vert Y\right\Vert _{\sigma}$
is also a direct consequence of Chebyshev's sum inequality in its
continuous form, which states that $\int_{0}^{1}f\left(u\right)\mathrm{d}u\cdot\int_{0}^{1}g\left(u\right)\mathrm{d}u\le\int_{0}^{1}f\left(u\right)g\left(u\right)\mathrm{d}u$
whenever $f$ and $g$ are both nondecreasing (choose $f=\sigma$
and $g=F_{\left|Y\right|}^{-1}$; cf.~\cite{Hardy1988}).\end{rem}
\begin{thm}[Comparability of $L_{\sigma}$\nobreakdash-spaces]
Suppose that 
\begin{equation}
c:=\sup_{0\le\alpha<1}\frac{\int_{\alpha}^{1}\sigma_{2}\left(u\right)\mathrm{d}u}{\int_{\alpha}^{1}\sigma_{1}\left(u\right)\mathrm{d}u}\label{eq:9}
\end{equation}
is finite ($c<\infty$), then 
\begin{equation}
\left\Vert Y\right\Vert _{\sigma_{2}}\le c\cdot\left\Vert Y\right\Vert _{\sigma_{1}}\qquad\left(Y\in L_{\sigma_{1}}\right)\label{eq:10}
\end{equation}
 and $L_{\sigma_{1}}\subset L_{\sigma_{2}}$; $c$ is moreover the
smallest constant satisfying \eqref{eq:10}, the identity 
\[
\id:\left(L_{\sigma_{1}},\left\Vert \cdot\right\Vert _{\sigma_{1}}\right)\rightarrow\left(L_{\sigma_{2}},\left\Vert \cdot\right\Vert _{\sigma_{2}}\right)
\]
 thus is continuous with norm $\left\Vert \id\right\Vert =c$. 
\end{thm}

\begin{proof}
To accept \eqref{eq:10} define the functions $S_{i}\left(\alpha\right):=\int_{\alpha}^{1}\sigma_{i}\left(u\right)\mathrm{d}u$
($i=1,\,2$), then by Riemann\textendash{}Stieltjes integration by
parts and as $u\mapsto F_{\left|Y\right|}^{-1}\left(u\right)$ is
nondecreasing, 
\begin{align*}
\left\Vert Y\right\Vert _{\sigma_{2}} & =\int_{0}^{1}F_{\left|Y\right|}^{-1}\left(u\right)\sigma_{2}\left(u\right)\mathrm{d}u=-\int_{0}^{1}F_{\left|Y\right|}^{-1}\left(u\right)\mathrm{d}S_{2}\left(u\right)\\
 & =-\left.F_{\left|Y\right|}^{-1}\left(u\right)S_{2}\left(u\right)\right|_{0}^{1}+\int_{0}^{1}S_{2}\left(u\right)\mathrm{d}F_{\left|Y\right|}^{-1}\left(u\right)=F_{\left|Y\right|}^{-1}\left(0\right)+\int_{0}^{1}S_{2}\left(u\right)\mathrm{d}F_{\left|Y\right|}^{-1}\left(u\right)\\
 & \le F_{\left|Y\right|}^{-1}\left(0\right)+c\cdot\int_{0}^{1}S_{1}\left(u\right)\mathrm{d}F_{\left|Y\right|}^{-1}\left(u\right)\\
 & =F_{\left|Y\right|}^{-1}\left(0\right)+c\cdot\left.F_{\left|Y\right|}^{-1}\left(u\right)S_{1}\left(u\right)\right|_{0}^{1}-c\cdot\int_{0}^{1}F_{\left|Y\right|}^{-1}\left(u\right)\mathrm{d}S_{1}\left(u\right)\\
 & =-F_{\left|Y\right|}^{-1}\left(0\right)\left(c-1\right)+c\cdot\int_{0}^{1}F_{\left|Y\right|}^{-1}\left(u\right)\sigma_{1}\left(u\right)\mathrm{d}u\le c\cdot\left\Vert Y\right\Vert _{\sigma_{1}},
\end{align*}
because $F_{\left|Y\right|}^{-1}\left(0\right)\ge0$ and $c\ge1$
(choose $\alpha=0$ in \eqref{eq:9}). 

To accept that $c$ is the smallest constant satisfying \eqref{eq:10}
just consider the random variable $Y=\one_{A^{c}}$, for which $\left\Vert Y\right\Vert _{\sigma}=\rho_{\sigma}\left(\one_{A^{c}}\right)=\int_{P\left(A\right)}^{1}\sigma\left(u\right)\mathrm{d}u$.
The assertion follows, as the measurable set $A$ may be chosen arbitrarily.
\end{proof}
It is a particular consequence of \eqref{eq:10} that 
\[
\AVaR_{\alpha_{1}}\left(\left|Y\right|\right)\le\AVaR_{\alpha_{2}}\left(\left|Y\right|\right)\le\frac{1-\alpha_{1}}{1-\alpha_{2}}\AVaR_{\alpha_{1}}\left(\left|Y\right|\right),
\]
which holds whenever $\alpha_{1}\le\alpha_{2}<1$. It should be noted,
however, that $\AVaR_{\alpha_{1}}\left(Y\right)\le\AVaR_{\alpha_{2}}\left(Y\right)\not\le\frac{1-\alpha_{1}}{1-\alpha_{2}}\AVaR_{\alpha_{1}}\left(Y\right)$
in general.

\medskip{}

The following representation result for spectral risk measures is
well known for $\sigma$ in an appropriate space. We extend it to
$L_{\sigma}$, the result will be used in the sequel.
\begin{prop}[Representation of the spectral risk measure]
\label{prop:Representation}$\rho_{\sigma}$ has the equivalent representation~%
\footnote{A random variable $U$ is uniformly distributed if $P\left(U\le u\right)=u$
whenever $u\in\left[0,\,1\right]$.%
} 
\begin{equation}
\rho_{\sigma}\left(Y\right)=\sup\left\{ \mathbb{E}\, Y\cdot\sigma\left(U\right):\, U\text{ is uniformly distributed}\right\} \label{eq:11}
\end{equation}
on $L_{\sigma}$.\end{prop}
\begin{rem}
For the Average Value-at-Risk it holds in particular that 
\begin{equation}
\AVaR_{\alpha}\left(Y\right)=\sup\left\{ \mathbb{E}\, Y\cdot Z:\,\mathbb{E}\, Z=1,\,0\le Z\le\frac{1}{1-\alpha}\right\} \label{eq:12}
\end{equation}
in view of the spectral function \eqref{eq:AVaR}.\end{rem}
\begin{proof}
Consider the random variable $Z=\sigma\left(U\right)$ for a uniformly
distributed random variable $U$, then $P\left(Z\le\sigma\left(\alpha\right)\right)=P\left(\sigma\left(U\right)\le\sigma\left(\alpha\right)\right)\ge P\left(U\le\alpha\right)=\alpha$,
that is $\VaR_{\alpha}\left(Z\right)\ge\sigma\left(\alpha\right)$.
But as $1=\int_{0}^{1}\sigma\left(\alpha\right)\mathrm{d}\alpha\le\int_{0}^{1}\VaR_{\alpha}\left(\sigma\left(U\right)\right)\mathrm{d}\alpha=\mathbb{E}\,\sigma\left(U\right)=\int_{0}^{1}\sigma(p)\,\mathrm{d}p=1$
it follows that 
\begin{equation}
\VaR_{\alpha}\left(Z\right)=\sigma\left(\alpha\right).\label{eq:13}
\end{equation}
Now $F_{Y}^{-1}\left(\cdot\right)$ is an increasing function, and
so is $\sigma\left(\cdot\right)$. By the Hardy\textendash{}Littlewood
rearrangement inequality (cf. \cite{Hoeffding} and \cite[Proposition 1.8]{PflugRomisch2007}
for the respective rearrangement inequality, sometimes also referred
to as \emph{Hardy}\textendash{}\emph{Littlewood}\textendash{}\emph{Pólya
inequality}, cf.~\cite{Dana2005}) it follows thus that 
\[
\mathbb{E}\, Y\cdot\sigma\left(U\right)\le\int_{0}^{1}F_{Y}^{-1}\left(\alpha\right)\sigma\left(\alpha\right)\mathrm{d}\alpha.
\]
However, if $Y$ and $U$ are coupled in a co-monotone way, then equality
is attained, that is $\mathbb{E}\, Y\cdot\sigma\left(U\right)=\int_{0}^{1}F_{Y}^{-1}\left(\alpha\right)\sigma\left(\alpha\right)\mathrm{d}\alpha$.
This proves the statement in view of the definition of the spectral
risk measure,~\eqref{eq:SRM}.
\end{proof}
The next theorem demonstrates that the spaces $L_{\sigma}$ really
add something to $L^{p}$ spaces, the space $L_{\sigma}$ is \emph{strictly
larger} than $L^{p}$.
\begin{thm}[$L_{\sigma}$ is larger than $L^{p}$]
\label{thm:LInfty}The following hold true:
\begin{enumerate}
\item Suppose that $\sigma\in L^{q}$ for some $1\le q<\infty$. Then the
space of natural domain $L_{\sigma}$ is strictly larger than $L^{p}$,
$L^{p}\subsetneqq L_{\sigma}$ ($\frac{1}{p}+\frac{1}{q}=1$).
\item In particular the space of natural domain $L_{\sigma}$ is (always)
strictly larger than $L^{\infty}$, $L^{\infty}\subsetneqq L_{\sigma}$
($q=1$). 
\end{enumerate}
\end{thm}
\begin{rem}
It should be noted that the statement of the latter theorem does not
hold for $\sigma\in L^{\infty}$: In this situation $\rho_{\sigma}$
is well defined on $L^{1}$, and $L_{\sigma}=L^{1}$ by the preceding
Theorem~\ref{thm:Inclusions}, \ref{enu:Lp}.\end{rem}
\begin{proof}
To prove the first assertion assume that $\sigma\in L^{q}$ for $1<q<\infty$.
Consider the uniquely determined numbers $t_{0}:=0<t_{1}<t_{2}<\dots<1$
for which $\int_{0}^{t_{n}}\sigma(u)^{q}\mathrm{d}u=\frac{\left\Vert \sigma\right\Vert _{q}^{q}}{\zeta\left(p+1\right)}\sum_{j=1}^{n}\frac{1}{j^{p+1}}$
and observe that $\int_{t_{n-1}}^{t_{n}}\sigma(u)^{q}\mathrm{d}u=\frac{\left\Vert \sigma\right\Vert _{q}^{q}}{\zeta\left(p+1\right)}\frac{1}{n^{p+1}}$.%
\footnote{$\zeta\left(p\right):=\sum_{n=1}^{\infty}\frac{1}{n^{p}}$ is Riemann's
Zeta function, the series converges whenever $p>1$.%
}  Define the function 
\[
\tau\left(u\right):=\begin{cases}
n & \text{ if }\end{cases}t_{n-1}\le u<t_{n},
\]
let $U$ be uniformly distributed and consider the random variable
\begin{equation}
Y:=\sigma\left(U\right)^{q-1}\cdot\tau\left(U\right).\label{eq:14}
\end{equation}
Note, by \eqref{eq:11}, that 
\begin{eqnarray*}
\rho_{\sigma}\left(Y\right) & = & \mathbb{E}\,\sigma\left(U\right)Y=\mathbb{E}\,\sigma\left(U\right)\sigma\left(U\right)^{q-1}\tau\left(U\right)=\mathbb{E}\,\sigma\left(U\right)^{q}\tau\left(U\right)\\
 & = & \int_{0}^{1}\sigma\left(u\right)^{q}\tau\left(u\right)\,\mathrm{d}u=\sum_{n=1}^{\infty}\int_{t_{n-1}}^{t_{n}}\sigma\left(u\right)^{q}\cdot n\,\mathrm{d}u\\
 & = & \frac{\left\Vert \sigma\right\Vert _{q}^{q}}{\zeta\left(p+1\right)}\sum_{n=1}^{\infty}\frac{n}{n^{p+1}}=\frac{\left\Vert \sigma\right\Vert _{q}^{q}}{\zeta\left(p+1\right)}\sum_{n=1}^{\infty}\frac{1}{n^{p}}=\left\Vert \sigma\right\Vert _{q}^{q}\frac{\zeta\left(p\right)}{\zeta\left(p+1\right)}<\infty,
\end{eqnarray*}
because $p>1$. Next, 
\begin{eqnarray*}
\left\Vert Y\right\Vert _{p}^{p} & = & \mathbb{E}\left|Y\right|^{p}=\int_{0}^{1}\sigma\left(u\right)^{\left(q-1\right)p}\tau\left(u\right)^{p}\mathrm{d}u\\
 & = & \int_{0}^{1}\sigma\left(u\right)^{q}\tau\left(u\right)^{p}\mathrm{d}u=\sum_{n=1}^{\infty}\int_{t_{n-1}}^{t_{n}}\sigma\left(u\right)^{q}\cdot n^{p}\,\mathrm{d}u\\
 & = & \frac{\left\Vert \sigma\right\Vert _{q}^{q}}{\zeta\left(p+1\right)}\sum_{n=1}^{\infty}\frac{n^{p}}{n^{p+1}}=\frac{\left\Vert \sigma\right\Vert _{q}^{q}}{\zeta\left(p+1\right)}\sum_{n=1}^{\infty}\frac{1}{n}=\infty.
\end{eqnarray*}
Hence, $Y\in L_{\sigma}$, but $Y\notin L^{p}$.

\medskip{}
The second statement of the theorem is actually the first statement
with $q=1$, but the above proof needs a modification: To accept it
define, as above, an increasing sequence of values by $t_{0}:=0<t_{1}<t_{2}<\dots<1$
satisfying $\int_{0}^{t_{n}}\sigma(t)\mathrm{d}t\ge1-2^{-n}$. Note,
that 
\[
\int_{t_{n-1}}^{t_{n}}\sigma(u)\mathrm{d}u\le\int_{t_{n-1}}^{1}\sigma\left(u\right)\mathrm{d}u=1-\int_{0}^{t_{n-1}}\sigma(u)\mathrm{d}u\le2^{1-n}.
\]
Define moreover the increasing function 
\[
\tau\left(\cdot\right):=\sum_{n=0}\one_{\left[t_{n},\,1\right]}\left(\cdot\right)
\]
(i.e. $\tau\left(t\right)=n$ if $t_{n-1}\le t<t_{n}$) and observe
that $\tau\nearrow\infty$ whenever $t\to1$. 

Now let $U$ be a uniformly distributed random variable and set $Y:=\tau\left(U\right)$.
Then 
\begin{align*}
\rho_{\sigma}\left(Y\right)= & \int_{0}^{1}\sigma(u)\tau(u)\mathrm{d}u=\sum_{n=1}\int_{t_{n-1}}^{t_{n}}\sigma(u)\tau(u)\mathrm{d}u\\
 & =\sum_{n=1}n\cdot\int_{t_{n-1}}^{t_{n}}\sigma(u)\mathrm{d}u\le\sum_{n=1}n\cdot2^{1-n}=4<\infty,
\end{align*}
so $Y\in L_{\sigma}$. But $Y\notin L^{\infty}$, because $P\left(Y\ge n\right)\ge1-t_{n-1}>0$
by definition of $\tau$. \end{proof}
\begin{rem}
Notably the preceding proof applies for the random variable $Y=\sigma\left(U\right)^{q-1}\cdot\tau\left(U\right)^{\alpha}$
in \eqref{eq:14} equally well whenever $1\le\alpha<p$, such that
$L_{\sigma}$ is larger than $L^{p}$ by an entire infinite dimensional
manifold.

\medskip{}

\end{rem}
It was demonstrated above that the space $L_{\sigma}$ is contained
in $L^{1}$. The above inequality \eqref{eq:sigmaLp}, $\left\Vert \cdot\right\Vert _{1}\le\left\Vert \cdot\right\Vert _{\sigma}$,
allows to prove an even much stronger result: a finite valued risk
measure cannot be considered on a space larger than $L^{1}$. This
is the content of the following theorem, which was communicated to
the author by Prof. Alexander Shapiro (Georgia Tech). In brief: it
does not make sense to consider risk measures on a space larger than
$L^{1}$.
\begin{thm}
\label{thm:Shapiro}Let $L\subset L^{0}$ be a vector space collecting
$\mathbb{R}$\nobreakdash-valued random variables on $\left(\left[0,1\right],\,\mathscr{B},\,\lambda\right)$
(the standard probability space equipped with its Borel sets) such
that $L\supsetneqq L^{1}$ and $\left|Y\right|\in L$, if $Y\in L$.
Then there does not exist a version independent, finite valued risk
measure on $L$. \end{thm}
\begin{proof}
Suppose that $\rho:\, L\to\mathbb{R}$ is a version independent, and
finite valued risk measure on $L$. Restricted to $L^{\infty}$, Kusuoka's
theorem (Theorem~\ref{thm:Kusuoka}) applies and $\rho$ takes the
form $\rho\left(\cdot\right)=\sup_{\sigma\in\mathscr{S}}\rho_{\sigma}\left(\cdot\right)$.
Choose $Y\in L\backslash L^{1}$, that is $\mathbb{E}\left|Y\right|=\infty$,
or  $\int_{0}^{p}F_{\left|Y\right|}^{-1}\left(u\right)\mathrm{d}u\to\infty$
whenever $p\to1$.

Next, pick any $\sigma\in\mathscr{S}$. Define $Y_{n}:=\min\left\{ n,\,\left|Y\right|\right\} $
and observe that $\rho\left(Y_{n}\right)\le\rho\left(\left|Y\right|\right)$
by monotonicity. Note that $Y_{n}\in L^{\infty}$ and hence, by Kusuoka's
representation, \eqref{eq:sigmaLp} and the particular choice of $Y$,
\[
\rho\left(\left|Y\right|\right)\ge\rho\left(Y_{n}\right)\ge\rho_{\sigma}\left(Y_{n}\right)=\left\Vert Y_{n}\right\Vert _{\sigma}\ge\left\Vert Y_{n}\right\Vert _{1}\ge\int_{0}^{P\left(\left|Y\right|\le n\right)}F_{\left|Y\right|}^{-1}\left(u\right)\,\mathrm{d}u\to\infty,
\]
as $n\rightarrow\infty$. Hence, $\rho$ is not finite valued on $L$. \end{proof}
\begin{thm}
$\left(L_{\sigma},\,\left\Vert \cdot\right\Vert _{\sigma}\right)$
is a Banach space over $\mathbb{R}$.\end{thm}
\begin{proof}
It remains to be shown that $\left(L_{\sigma},\,\left\Vert \cdot\right\Vert _{\sigma}\right)$
is complete. For this let $\left(Y_{k}\right)_{k}$ be a Cauchy sequence
for $\left\Vert \cdot\right\Vert _{\sigma}$. By \eqref{eq:sigmaLp}
the sequence $\left(Y_{k}\right)_{k}$ is a Cauchy sequence for $\left\Vert \cdot\right\Vert _{1}$
as well, and from completeness of $L^{1}$ it follows that there exists
a limit $Y\in L^{1}$. We shall show that $Y\in L_{\sigma}$.

It follows from convergence in $L^{1}$ that $\left(Y_{k}\right)_{k}$
converges in distribution, that is $F_{Y_{k}}\left(y\right)\rightarrow F_{Y}\left(y\right)$
for every point $y$ where $F_{Y}$ is continuous and moreover $F_{\left|Y_{k}\right|}^{-1}\left(\cdot\right)\rightarrow F_{\left|Y\right|}^{-1}\left(\cdot\right)$
(cf. \cite[Chapter 21]{vdVaart}). Now 
\begin{eqnarray*}
\left\Vert Y\right\Vert _{\sigma} & = & \rho_{\sigma}\left(\left|Y\right|\right)=\int_{0}^{1}\sigma\left(t\right)F_{\left|Y\right|}^{-1}\left(t\right)\mathrm{d}t=\int_{0}^{1}\sigma\left(t\right)\lim_{k\to\infty}F_{\left|Y_{k}\right|}^{-1}\left(t\right)\mathrm{d}t\\
 & = & \int_{0}^{1}\liminf_{k\to\infty}\sigma\left(t\right)F_{\left|Y_{k}\right|}^{-1}\left(t\right)\mathrm{d}t\le\liminf_{k\to\infty}\int_{0}^{1}\sigma\left(t\right)F_{\left|Y_{k}\right|}^{-1}\left(t\right)\mathrm{d}t=\liminf_{k\to\infty}\left\Vert Y_{k}\right\Vert _{\sigma}
\end{eqnarray*}
by Fatou's Lemma, which is applicable because $F_{\left|Y_{k}\right|}^{-1}\left(\cdot\right)\ge0$. 

As $\left(Y_{k}\right)_{k}$ is a Cauchy sequence one may pick $k^{*}\in\mathbb{N}$
such that $\left\Vert Y_{k}-Y_{k^{*}}\right\Vert _{\sigma}<1$ for
all $k>k^{*}$, and hence $\left\Vert Y_{k}\right\Vert _{\sigma}\le\left\Vert Y_{k^{*}}\right\Vert _{\sigma}+\left\Vert Y_{k}-Y_{k^{*}}\right\Vert _{\sigma}<\left\Vert Y_{k^{*}}\right\Vert _{\sigma}+1<\infty$
by the triangle inequality. The sequence $\left(Y_{k}\right)_{k}$
thus is uniformly bounded in its norm. Hence, 
\[
\left\Vert Y\right\Vert _{\sigma}\le\liminf_{k\to\infty}\left\Vert Y_{k}\right\Vert _{\sigma}\le\left\Vert Y_{k^{*}}\right\Vert _{\sigma}+1<\infty,
\]
 that is $Y\in L_{\sigma}$ and $L_{\sigma}$ thus is complete.\end{proof}
\begin{example}
Consider the spectrum $\sigma\left(\alpha\right)=\frac{1}{2\sqrt{1-\alpha}}$.
It should be noted that $L_{\sigma}\supset\bigcup_{p>2}L^{p}$, and
$\left\Vert \cdot\right\Vert _{\sigma}$ provides a reasonable norm
on that set. 

Restricted to $L^{p}$, for some $p>2$, the open mapping theorem
(cf. \cite{Rudin1973} or \cite{Aliprantis2006}) insures that the
norms are equivalent, that is there are constants $c_{1}$ and $c_{2}$
such that 
\[
c_{1}\cdot\left\Vert Y\right\Vert _{p}\le\left\Vert Y\right\Vert _{\sigma}\le c_{2}\cdot\left\Vert Y\right\Vert _{p}\qquad(Y\in L^{p}\subset L_{\sigma}).
\]
The latter inequalities hold just for $Y\in L^{p}$, but not for $Y\in L_{\sigma}$. \end{example}
\begin{prop}
Measurable, simple (step) functions are dense in $L_{\sigma}$, and
in particular $L^{\infty}$ is dense in $L_{\sigma}$. \end{prop}
\begin{proof}
Given $Y\in L_{\sigma}$ and $\varepsilon>0$, find $t_{0}\in\left(0,1\right)$
such that $\int_{0}^{t_{0}}F_{Y}^{-1}(u)\sigma(u)\mathrm{d}u<\frac{\varepsilon}{3}$
and set $s\left(t\right):=F_{Y}^{-1}\left(t_{0}\right)$ whenever
$t\le t_{0}$. Moreover, find $t_{1}\in\left(0,1\right)$ such that
$\int_{t_{1}}^{1}F_{Y}^{-1}(u)\sigma(u)\mathrm{d}u<\frac{\varepsilon}{3}$
and set $s\left(t\right):=F_{Y}^{-1}\left(t_{1}\right)$ whenever
$t\ge t_{1}$. In between, as $F_{Y}^{-1}(t)$ is nondecreasing on
the compact $\left[t_{0},t_{1}\right]$, there is an increasing step
function $s\left(t\right)$ such that $\left|s(t)-F_{Y}^{-1}(t)\right|\sigma(t)<\frac{\varepsilon}{3}$.
Let $U$ be uniformly distributed and co-monotone with $Y$. Then
it holds that $\left\Vert Y-s\left(U\right)\right\Vert _{\sigma}<\varepsilon$
by construction of the step function $s$.
\end{proof}

\section{\label{sec:LDual}The Dual of the natural domain $L_{\sigma}$}

Risk measures are convex and lower semi-continuous (cf. \cite{Schachermayer2006})
functions, hence they have a dual representation by involving the
Fenchel\textendash{}Moreau Theorem (also Legendre transformation,
see below). This representation involves the dual space in a natural
way, and hence it is of interest to understand the dual of the Banach
space $\left(L_{\sigma},\,\left\Vert \cdot\right\Vert _{\sigma}\right)$.
We describe the norm of the dual and identify the dual with a subspace
of $L^{1}$. The respective results are proven in this section, moreover
essential properties of the dual are highlighted.
\begin{thm}[Fenchel\textendash{}Moreau]
Let $\mathscr{Y}$ be a Banach space and $f:\,\mathscr{Y}\to\mathbb{R}\cup\left\{ \infty\right\} $
be convex and lower semi-continuous with $f\left(Y_{0}\right)<\infty$
for an $Y_{0}\in\mathscr{Y}$. Then 
\[
f^{**}=f,
\]
where 
\[
f^{*}\left(Z^{*}\right):=\sup_{Y\in\mathscr{Y}}Z^{*}\left(Y\right)-f\left(Y\right)\quad\text{and }\quad f^{**}\left(Y\right):=\sup_{Z^{*}\in\mathscr{Y}^{*}}Z^{*}\left(Y\right)-f^{*}\left(Z^{*}\right).
\]
\end{thm}
\begin{proof}
cf. \cite{Rockafellar1974}.
\end{proof}
Note, that a risk measure $\rho_{\sigma}$ is not only lower semicontinuous,
by Proposition~\ref{prop:Continuity} it is continuous with respect
to the norm $\left\Vert \cdot\right\Vert _{\sigma}$ on the Banach
space $\mathscr{Y}=\left(L_{\sigma},\,\left\Vert \cdot\right\Vert _{\sigma}\right)$.
By the Fenchel\textendash{}Moreau theorem thus $\rho_{\sigma}^{**}=\rho_{\sigma}$.
To involve it on its natural domain $\mathscr{Y}=\left(L_{\sigma},\,\left\Vert \cdot\right\Vert _{\sigma}\right)$
its dual $\mathscr{Y}^{*}=\left(L_{\sigma},\,\left\Vert \cdot\right\Vert _{\sigma}\right)^{*}$
has to be available, and this is elaborated in the sequel.
\begin{defn}
For a spectral function $\sigma$ and a random variable $Z\in L^{1}$
define the binary relation 
\begin{equation}
Z\preccurlyeq\sigma\;\text{ iff }\quad\AVaR_{\alpha}\left(\left|Z\right|\right)\le\frac{1}{1-\alpha}\int_{\alpha}^{1}\sigma(u)\mathrm{d}u\text{ for all }0\le\alpha<1,\label{eq:15}
\end{equation}
the gauge function (Minkowski functional) 
\begin{align}
\left\Vert Z\right\Vert _{\sigma}^{*}: & =\inf\left\{ \eta\ge0:\,\AVaR_{\alpha}\left(\left|Z\right|\right)\le\frac{\eta}{1-\alpha}\int_{\alpha}^{1}\sigma(u)\mathrm{d}u\text{ for all }0\le\alpha<1\right\} \label{eq:16}\\
 & =\inf\left\{ \eta\ge0:\,\left|Z\right|\preccurlyeq\eta\cdot\sigma\right\} \nonumber 
\end{align}
and the set $L_{\sigma}^{*}:=\left\{ Z\in L^{1}:\,\left\Vert Z\right\Vert _{\sigma}^{*}<\infty\right\} $.
\end{defn}
It should be noted that the relation \eqref{eq:15}, which is a kind
of second order stochastic dominance relation (cf.~\cite{Denuit,DentchevaRusz2004}),
can be interpreted as a growth condition for $\left|Z\right|$, which
is a condition on $Z$'s tails: $Z\preccurlyeq\eta\cdot\sigma$ can
only hold true if $\left|Z\right|$ does not grow (in quantiles) faster
towards $\infty$ than $\eta\cdot\sigma$. 

Notice as well that 
\begin{equation}
\left\Vert Z\right\Vert _{\sigma}^{*}\le\eta\text{ if and only if }\AVaR_{\alpha}\left(\left|Z\right|\right)\le\frac{\eta}{1-\alpha}\int_{\alpha}^{1}\sigma(u)\mathrm{d}u\text{ for all }0\le\alpha<1.\label{eq:17}
\end{equation}

Moreover the functions $\alpha\mapsto\int_{\alpha}^{1}\sigma(u)\mathrm{d}u$
and $\alpha\mapsto\left(1-\alpha\right)\AVaR_{\alpha}\left(\left|Z\right|\right)$
are both continuous functions on $\left[0,\,1\right]$, so the maximum
of their difference is attained in $\left[0,\,1\right]$. Hence, the
infimum in \eqref{eq:16} will be attained as well at some $\eta\ge0$.
\begin{example}
For $U$ a uniformly distributed random variable it follows readily
from the definition and \eqref{eq:13} that 
\begin{equation}
\left\Vert \sigma\left(U\right)\right\Vert _{\sigma}^{*}=1.\label{eq:25}
\end{equation}

The norm of the indicator function has the explicit form 
\begin{equation}
\left\Vert \one_{A}\right\Vert _{\sigma}^{*}=\frac{1}{\frac{1}{P\left(A\right)}\int_{1-P\left(A\right)}^{1}\sigma\left(u\right)\mathrm{d}u},\label{eq:12-1}
\end{equation}
which derives from $\AVaR_{\alpha}\left(\one_{A}\right)=\min\left\{ 1,\,\frac{P\left(A\right)}{1-\alpha}\right\} $
and the particular choice $\alpha=1-P\left(A\right)$ in \eqref{eq:15}.
Immediate consequences of~\eqref{eq:12-1} are further the bounds
$P\left(A\right)\le\left\Vert \one_{A}\right\Vert _{\sigma}^{*}\le1$. \end{example}
\begin{rem}
Given Kusuoka's representation one may employ the measure $\mu$ directly
instead of the spectral density $\sigma$ by involving \eqref{eq:Sigma}.
It holds that 
\[
\frac{1}{1-\alpha}\int_{\alpha}^{1}\sigma_{\mu}(u)\mathrm{d}u=\int_{0}^{1}\min\left\{ \frac{1}{1-u},\,\frac{1}{1-\alpha}\right\} \mathrm{d}\mu\left(u\right),
\]
the condition $Z\preccurlyeq\sigma_{\mu}$ thus reads directly 
\[
Z\preccurlyeq\sigma_{\mu}\;\text{ iff }\quad\AVaR_{\alpha}\left(\left|Z\right|\right)\le\int_{0}^{1}\min\left\{ \frac{1}{1-u},\,\frac{1}{1-\alpha}\right\} \mathrm{d}\mu\left(u\right)\text{ for all }0\le\alpha<1.
\]
Notice as well that $\int_{0}^{1}\min\left\{ \frac{1}{1-u},\,\frac{1}{1-\alpha}\right\} \mathrm{d}\mu\left(u\right)$
represents an expectation of a (bounded) function  with respect to
the measure $\mu$.\end{rem}
\begin{lem}
The unit ball of the norm $\left\Vert \cdot\right\Vert _{\sigma}^{*}$
is 
\[
B_{\sigma}=\left\{ Z\in L^{1}:\,\AVaR_{\alpha}\left(\left|Z\right|\right)\le\frac{1}{1-\alpha}\int_{\alpha}^{1}\sigma(u)\mathrm{d}u\text{ for all }0\le\alpha<1\right\} ,
\]
which is an absolutely convex set.\end{lem}
\begin{proof}
Just observe that 
\begin{eqnarray*}
\AVaR_{\alpha}\left(\left|\lambda_{1}Z_{1}+\lambda_{2}Z_{2}\right|\right) & \le & \AVaR_{\alpha}\left(\left|\lambda_{1}Z_{1}\right|+\left|\lambda_{2}Z_{2}\right|\right)\\
 & = & 2\cdot\AVaR_{\alpha}\left(\frac{1}{2}\left|\lambda_{1}Z_{1}\right|+\frac{1}{2}\left|\lambda_{2}Z_{2}\right|\right)\\
 & \le & \left|\lambda_{1}\right|\AVaR_{\alpha}\left(\left|Z_{1}\right|\right)+\left|\lambda_{2}\right|\AVaR_{\alpha}\left(\left|Z_{2}\right|\right)
\end{eqnarray*}
by monotonicity, convexity and positive homogeneity (sub-additivity).
For $Z_{1},\, Z_{2}\in B_{\sigma}$ and $\left|\lambda_{1}\right|+\left|\lambda_{2}\right|\le1$
it follows thus that $\lambda_{1}Z_{1}+\lambda_{2}Z_{2}\in B_{\sigma}$
and $B_{\sigma}$ is absolutely convex.
\end{proof}

\paragraph{Monotonicity. }

It follows from monotonicity of the Average Value-at-Risk that 
\begin{equation}
\left\Vert Y_{1}\right\Vert _{\sigma}^{*}\le\left\Vert Y_{2}\right\Vert _{\sigma}^{*},\text{ if }\left|Y_{1}\right|\le\left|Y_{2}\right|.\label{eq:Monotonicity}
\end{equation}

\paragraph{Comparison with $L^{1}$. }

For $Z\in L^{\sigma}$, $\left\Vert Z\right\Vert _{\sigma}^{*}\le\eta$
implies that $\mathbb{E}\left|Z\right|\le\eta$ (by the choice $\alpha=0$
in~\eqref{eq:17}), hence 
\begin{equation}
\left\Vert Z\right\Vert _{1}\le\left\Vert Z\right\Vert _{\sigma}^{*}\label{eq:18}
\end{equation}
and $L_{\sigma}^{*}\subset L^{1}$.

\paragraph{Comparison with $L^{\infty}$. }

Suppose that $\sigma$ is bounded and $Z\in L^{\infty}$. Then $\AVaR_{\alpha}\left(\left|Z\right|\right)\to\left\Vert Z\right\Vert _{\infty}$
and $\frac{1}{1-\alpha}\int_{\alpha}^{1}\sigma(u)\mathrm{d}u\to\left\Vert \sigma\right\Vert _{\infty}$,
as $\alpha\to1$, and consequently $\left\Vert Z\right\Vert _{\infty}\le\eta\cdot\left\Vert \sigma\right\Vert _{\infty}$
has to hold by \eqref{eq:17} for $\eta$ to be feasible. That is,
\begin{equation}
\left\Vert Z\right\Vert _{\infty}\le\left\Vert Z\right\Vert _{\sigma}^{*}\cdot\left\Vert \sigma\right\Vert _{\infty}.\label{eq:19}
\end{equation}

\paragraph{Upper bound. }

An upper bound for the norm $\left\Vert \cdot\right\Vert _{\sigma}^{*}$
is given by 
\[
\left\Vert Z\right\Vert _{\sigma}^{*}\le\sup_{0\le u<1}\frac{F_{\left|Z\right|}^{-1}(u)}{\sigma(u)},
\]
where the conventions $\frac{0}{0}=0$ and $\frac{1}{0}=\infty$ have
to be employed. Indeed, if $\frac{F_{\left|Z\right|}^{-1}(u)}{\sigma(u)}\le\eta$,
then integrating gives $\left(1-\alpha\right)\AVaR_{\alpha}\left(\left|Z\right|\right)=\int_{\alpha}^{1}F_{\left|Z\right|}^{-1}(u)\mathrm{d}u\le\eta\cdot\int_{\alpha}^{1}\sigma\left(u\right)\mathrm{d}u$,
which in turn means that $\left\Vert Z\right\Vert _{\sigma}^{*}\le\eta$.
Notice, however, that $Z\mapsto\sup_{0\le u<1}\frac{F_{\left|Z\right|}^{-1}(u)}{\sigma(u)}$
is not a norm, it does not satisfy the triangle inequality.

\paragraph{Simple functions. }

For $Z=\sum_{j=1}^{n}a_{j}\one_{A_{j}}$ a simple (step) function,
$\alpha\mapsto\left(1-\alpha\right)\AVaR_{\alpha}\left(\left|Z\right|\right)=\int_{0}^{1}F_{\left|Z\right|}^{-1}\left(u\right)\mathrm{d}u$
is piecewise linear. As $\alpha\mapsto\int_{\alpha}^{1}\sigma\left(u\right)\mathrm{d}u$
is concave (this is, because $\sigma$ is increasing), the defining
condition \eqref{eq:17} has to be verified on finite many points
only, such that simple functions are contained in $L_{\sigma}^{*}$. 
\begin{prop}
The pair $\left(L_{\sigma}^{*},\,\left\Vert \cdot\right\Vert _{\sigma}^{*}\right)$
is a Banach space.\end{prop}
\begin{proof}
Notice first that $\left\Vert Z\right\Vert _{\sigma}^{*}=0$ implies
that $\AVaR_{\alpha}\left(\left|Z\right|\right)=0$ for all $\alpha<1$,
so 
\[
0=\lim_{\alpha\nearrow1}\AVaR_{\alpha}\left(\left|Z\right|\right)=\esssup\left|Z\right|,
\]
that is $Z=0$ almost everywhere, such that $\left\Vert \cdot\right\Vert _{\sigma}^{*}$
separates points in $L_{\sigma}^{*}$. 

Positive homogeneity is immediate and inherited from the Average Value-at-Risk.

As for the triangle inequality let $\eta_{1}$ and $\eta_{2}$, resp.
satisfy \eqref{eq:16} for $Z_{1}$ and $Z_{2}$, resp.. Then, by
monotonicity and sub-additivity of the Average Value-at-Risk, 
\[
\AVaR_{\alpha}\left(\left|Z_{1}+Z_{2}\right|\right)\le\AVaR_{\alpha}\left(\left|Z_{1}\right|+\left|Z_{2}\right|\right)\le\AVaR_{\alpha}\left(\left|Z_{1}\right|\right)+\AVaR_{\alpha}\left(\left|Z_{2}\right|\right)
\]
such that 
\[
\AVaR_{\alpha}\left(\left|Z_{1}+Z_{2}\right|\right)\le\frac{\eta_{1}+\eta_{2}}{1-\alpha}\int_{\alpha}^{1}\sigma(u)\mathrm{d}u,
\]
that is finally $\left\Vert Z_{1}+Z_{2}\right\Vert _{\sigma}^{*}\le\left\Vert Z_{1}\right\Vert _{\sigma}^{*}+\left\Vert Z_{2}\right\Vert _{\sigma}^{*}$,
the triangle inequality.

Finally completeness remains to be shown. For this let $Z_{k}$ be
a Cauchy sequence. Hence there is a natural number $k^{*}$, such
that $\left\Vert Z_{k}\right\Vert _{\sigma}^{*}\le\left\Vert Z_{k^{*}}\right\Vert _{\sigma}^{*}+\left\Vert Z_{k}-Z_{k^{*}}\right\Vert _{\sigma}^{*}\le\left\Vert Z_{k^{*}}\right\Vert _{\sigma}^{*}+1$,
that is there is $\eta\ge0$ ($\eta$ satisfies $\eta\le\left\Vert Z_{k^{*}}\right\Vert _{\sigma}^{*}+1$)
such that 
\[
\AVaR_{\alpha}\left(\left|Z_{k}\right|\right)\le\frac{\eta}{1-\alpha}\int_{\alpha}^{1}\sigma(u)\mathrm{d}u
\]
for all $k>k^{*}$ and $\alpha\in\left(0,1\right)$. Next, by \eqref{eq:18}
$Z_{k}$ is a Cauchy sequence for $L^{1}$ as well, hence there is
a limit $Z\in L^{1}$, and $Z_{k}$ converges in distribution and
in quantiles. By Fatou's inequality, 
\begin{eqnarray*}
\AVaR_{\alpha}\left(\left|Z\right|\right) & = & \frac{1}{1-\alpha}\int_{\alpha}^{1}F_{\left|Z\right|}^{-1}(u)\mathrm{d}u=\frac{1}{1-\alpha}\int_{\alpha}^{1}\liminf_{k\to\infty}F_{\left|Z_{k}\right|}^{-1}(u)\mathrm{d}u\\
 & \le & \frac{1}{1-\alpha}\liminf_{k\to\infty}\int_{\alpha}^{1}F_{\left|Z_{k}\right|}^{-1}(u)\mathrm{d}u=\liminf_{k\to\infty}\AVaR_{\alpha}\left(\left|Z_{k}\right|\right)\\
 & \le & \frac{\eta}{1-\alpha}\cdot\int_{\alpha}^{1}\sigma(u)\mathrm{d}u.
\end{eqnarray*}
The limit $Z\in L^{1}$ thus satisfies the defining conditions to
qualify for $L_{\sigma}^{*}$ and $\left\Vert Z\right\Vert _{\sigma}^{*}\le\eta$.
It follows that $Z\in L_{\sigma}^{*}$ and $\left(L_{\sigma}^{*},\,\left\Vert \cdot\right\Vert _{\sigma}^{*}\right)$
thus is a Banach space.\end{proof}
\begin{thm}
The space $\left(L_{\sigma}^{*},\,\left\Vert \cdot\right\Vert _{\sigma}^{*}\right)$
is the dual of \textup{$\left(L_{\sigma},\,\left\Vert \cdot\right\Vert _{\sigma}\right)$.}\end{thm}
\begin{proof}
Let $Y\in L_{\sigma}$ and $Z\in L_{\sigma}^{*}$ with $\left\Vert Z\right\Vert _{\sigma}^{*}=:\eta$
be chosen. Then note that 
\[
\left|\mathbb{E}\, YZ\right|\le\mathbb{E}\left|Y\right|\cdot\left|Z\right|\le\int_{0}^{1}F_{\left|Y\right|}^{-1}\left(u\right)F_{\left|Z\right|}^{-1}\left(u\right)\mathrm{d}u
\]
 by the Hardy\textendash{}Littlewood\textendash{}Pólya inequality.
To abbreviate the notation we introduce the functions $S\left(u\right):=\int_{u}^{1}\sigma(p)\mathrm{d}p$
and $G\left(u\right):=\int_{u}^{1}F_{\left|Z\right|}^{-1}\left(p\right)\mathrm{d}p$
(the functions are well defined, because $\sigma\in L^{1}$ and $Z\in L^{1}$).
Then, by Riemann\textendash{}Stieltjes integration by parts, 
\begin{eqnarray*}
\int_{0}^{1}F_{\left|Y\right|}^{-1}\left(u\right)F_{\left|Z\right|}^{-1}\left(u\right)\mathrm{d}u & = & -\int_{0}^{1}F_{\left|Y\right|}^{-1}\left(u\right)\mathrm{d}G\left(u\right)\\
 & = & -\left.F_{\left|Y\right|}^{-1}\left(u\right)G\left(u\right)\right|_{u=0}^{1}+\int_{0}^{1}G\left(u\right)\mathrm{d}F_{\left|Y\right|}^{-1}\left(u\right)\\
 & = & F_{\left|Y\right|}^{-1}\left(0\right)\cdot\mathbb{E}\left|Z\right|+\int_{0}^{1}G\left(u\right)\mathrm{d}F_{\left|Y\right|}^{-1}\left(u\right).
\end{eqnarray*}
Now note that $F_{\left|Y\right|}^{-1}\left(\cdot\right)$ is an increasing
function, and $G\left(u\right)=\int_{u}^{1}F_{\left|Z\right|}^{-1}\left(p\right)\mathrm{d}p\le\eta\cdot\int_{u}^{1}\sigma\left(p\right)\mathrm{d}p=\eta\cdot S\left(u\right)$
because $\left\Vert Z\right\Vert _{\sigma}^{*}\le\eta$. Thus, and
employing again Riemann\textendash{}Stieltjes integration by parts,
\begin{eqnarray*}
\left|\mathbb{E}YZ\right| & \le & F_{\left|Y\right|}^{-1}\left(0\right)\cdot\left\Vert Z\right\Vert _{1}+\eta\cdot\int_{0}^{1}S\left(u\right)\mathrm{d}F_{\left|Y\right|}^{-1}\left(u\right)\\
 & = & F_{\left|Y\right|}^{-1}\left(0\right)\cdot\left\Vert Z\right\Vert _{1}+\eta\cdot\left.S\left(u\right)F_{\left|Y\right|}^{-1}\left(u\right)\right|_{u=0}^{1}-\eta\cdot\int_{0}^{1}F_{\left|Y\right|}^{-1}\left(u\right)\mathrm{d}S\left(u\right)\\
 & = & F_{\left|Y\right|}^{-1}\left(0\right)\cdot\left\Vert Z\right\Vert _{1}-\eta\cdot F_{\left|Y\right|}^{-1}\left(0\right)+\eta\cdot\int_{0}^{1}F_{\left|Y\right|}^{-1}\left(u\right)\sigma\left(u\right)\mathrm{d}u\\
 & = & F_{\left|Y\right|}^{-1}\left(0\right)\cdot\left(\left\Vert Z\right\Vert _{1}-\eta\right)+\eta\cdot\int_{0}^{1}F_{\left|Y\right|}^{-1}\left(u\right)\sigma\left(u\right)\mathrm{d}u\\
 & = & F_{\left|Y\right|}^{-1}\left(0\right)\cdot\left(\left\Vert Z\right\Vert _{1}-\left\Vert Z\right\Vert _{\sigma}^{*}\right)+\left\Vert Z\right\Vert _{\sigma}^{*}\cdot\int_{0}^{1}F_{\left|Y\right|}^{-1}\left(u\right)\sigma\left(u\right)\mathrm{d}u.
\end{eqnarray*}
Finally observe that $F_{\left|Y\right|}^{-1}\left(0\right)=\essinf\left|Y\right|\ge0$
and $\left\Vert Z\right\Vert _{1}-\left\Vert Z\right\Vert _{\sigma}^{*}\le0$
by \eqref{eq:18}, hence 
\begin{eqnarray*}
\left|\mathbb{E}\, YZ\right| & \le & \left\Vert Z\right\Vert _{\sigma}^{*}\cdot\int_{0}^{1}F_{\left|Y\right|}^{-1}\left(u\right)\sigma\left(u\right)\mathrm{d}u=\rho_{\sigma}\left(\left|Y\right|\right)\cdot\left\Vert Z\right\Vert _{\sigma}^{*}=\left\Vert Y\right\Vert _{\sigma}\cdot\left\Vert Z\right\Vert _{\sigma}^{*}.
\end{eqnarray*}
This proves that for every $Z\in L_{\sigma}^{*}$ the linear mapping
$Y\mapsto\mathbb{E}\, YZ$ is continuous with respect to the norm
$\left\Vert \cdot\right\Vert _{\sigma}$. 

\medskip{}

It remains to be shown that every linear, continuous mapping $\zeta$
in the dual of $L_{\sigma}$ ($\zeta\in\left(L_{\sigma},\,\left\Vert \cdot\right\Vert _{\sigma}\right)^{*}$)
takes the form $\zeta\left(Y\right)=\mathbb{E}YZ$ for some $Z\in L_{\sigma}^{*}$.
For this consider the (signed) measure $\mu\left(A\right):=\zeta\left(\one_{A}\right)$.
If $A=\bigcup_{i=1}^{\infty}A_{i}$ is a disjoint union of countably
measurable sets, then $\one_{A}=\sum_{i=1}^{\infty}\one_{A_{i}}$.
Clearly, 
\[
\left\Vert \one_{A}-\sum_{i=1}^{n}\one_{A_{i}}\right\Vert _{\sigma}=\int_{1-\sum_{i=n+1}^{\infty}P\left(A_{i}\right)}^{1}\sigma(u)\mathrm{d}u\xrightarrow[n\to\infty]{}0,
\]
as $P$ is sigma-finite and $\sigma\in L^{1}$. It follows by continuity
of $\zeta$ with respect to $\left\Vert \cdot\right\Vert _{\sigma}$
that 
\[
\mu\left(A\right)=\zeta\left(\one_{A}\right)=\zeta\left(\sum_{i=1}^{\infty}\one_{A_{i}}\right)=\sum_{i=1}^{\infty}\zeta\left(\one_{A_{i}}\right)=\sum_{i=1}^{\infty}\mu\left(A_{i}\right),
\]
hence $\mu$ is a sigma-finite measure. If $P\left(A\right)=0$, then
\[
\left|\mu\left(A\right)\right|=\left|\zeta\left(\one_{A}\right)\right|\le\left\Vert \zeta\right\Vert \cdot\left\Vert \one_{A}\right\Vert _{\sigma}=\left\Vert \zeta\right\Vert \cdot\int_{0}^{1}\sigma\left(u\right)F_{\one_{A}}^{-1}\left(u\right)\mathrm{d}u=0,
\]
because $F_{\one_{A}}^{-1}\left(u\right)=0$ for every $u<1$. It
follows that $\mu\left(A\right)=0$, such that $\mu$ is moreover
absolutely continuous with respect to $P$. 

Let $Z\in L^{0}$ be the Radon\textendash{}Nikodým derivative, $\mathrm{d}\mu=Z\mathrm{d}P$.
Then $\zeta\left(\one_{A}\right)=\mu\left(A\right)=\int_{A}Z\mathrm{d}P=\int Z\one_{A}\mathrm{d}P=\mathbb{E}\, Z\one_{A}$
and hence $\zeta\left(\phi\right)=\mathbb{E}\, Z\phi$ for all simple
functions $\phi$ by linearity and $\left|\mathbb{E}\, Z\phi\right|=\left|\zeta\left(\phi\right)\right|\le\left\Vert \zeta\right\Vert \cdot\left\Vert \phi\right\Vert _{\sigma}$
by continuity of $\zeta$. 

Choose the function $\phi:=\sign\, Z$ (a simple function) to see
that $\mathbb{E}\left|Z\right|\le\left\Vert \zeta\right\Vert $, that
is $Z\in L^{1}$. 

Note as well that $\mathbb{E}\left|Z\right|\phi=\mathbb{E}\, Z\cdot\sign\left(Z\right)\phi\le\left\Vert \zeta\right\Vert \cdot\left\Vert \sign\left(Z\right)\phi\right\Vert _{\sigma}\le\left\Vert \zeta\right\Vert \cdot\left\Vert \phi\right\Vert _{\sigma}$,
because $\rho_{\sigma}$ is monotone and $\left|\sign\left(Z\right)\cdot\phi\right|\le\left|\phi\right|$.
For any measurable set $A$ (with complement denoted $A^{c}$) thus
\[
\mathbb{E}\left|Z\right|\one_{A^{c}}\le\left\Vert \zeta\right\Vert \cdot\left\Vert \one_{A^{c}}\right\Vert _{\sigma}=\left\Vert \zeta\right\Vert \cdot\rho_{\sigma}\left(\one_{A^{c}}\right)=\left\Vert \zeta\right\Vert \cdot\int_{P\left(A\right)}^{1}\sigma(u)\mathrm{d}u,
\]
and hence $\mathbb{E}\left|Z\right|\frac{\one_{A^{c}}}{P\left(A^{c}\right)}\le\left\Vert \zeta\right\Vert \cdot\frac{1}{1-P\left(A\right)}\int_{P\left(A\right)}^{1}\sigma(u)\mathrm{d}u$.
Taking the supremum over all sets $A$ with $P\left(A\right)\le\alpha$
gives
\begin{eqnarray*}
\AVaR_{\alpha}\left(\left|Z\right|\right) & = & \sup_{P\left(A^{c}\right)\ge1-\alpha}\mathbb{E}\left|Z\right|\frac{\one_{A^{c}}}{P\left(A^{c}\right)}\le\left\Vert \zeta\right\Vert \cdot\sup_{P\left(A\right)\le\alpha}\frac{1}{1-P\left(A\right)}\int_{P\left(A\right)}^{1}\sigma(u)\mathrm{d}u\\
 & = & \frac{\left\Vert \zeta\right\Vert }{1-\alpha}\int_{\alpha}^{1}\sigma(u)\mathrm{d}u
\end{eqnarray*}
by \eqref{eq:12} and because $\sigma$ is increasing. It follows
that $\left\Vert Z\right\Vert _{\sigma}^{*}\le\left\Vert \zeta\right\Vert $
and thus $Z\in L_{\sigma}^{*}$. This completes the proof.
\end{proof}

\paragraph{The Hahn-Banach functional. }

Let $Y\in L_{\sigma}$ be fixed, and let $U$ be coupled in a co-monotone
way with $\left|Y\right|$. Define $Z_{Y}:=\sigma\left(U\right)\cdot\sign\, Y$
and observe that $F_{\sigma\left(U\right)}^{-1}\left(\alpha\right)=\sigma\left(\alpha\right)$
by \eqref{eq:13}. Hence $\AVaR_{\alpha}\left(\sigma\left(U\right)\right)=\frac{1}{1-\alpha}\int_{\alpha}^{1}\sigma\left(u\right)\mathrm{d}u$,
and it follows that $\left\Vert Z_{Y}\right\Vert _{\sigma}^{*}=1$.
On the other side $\mathbb{E}\, Y\cdot Z_{Y}=\mathbb{E}\,\left|Y\right|\cdot\sigma\left(U\right)=\int_{0}^{1}F_{\left|Y\right|}^{-1}\left(u\right)\sigma\left(u\right)\mathrm{d}u=\left\Vert Y\right\Vert _{\sigma}$.
$Z_{Y}$ thus is a maximizer of the problem 
\[
\left\Vert Y\right\Vert _{\sigma}=\max\left\{ \mathbb{E}\, Y\cdot Z:\,\left\Vert Z\right\Vert _{\sigma}^{*}\le1\right\} .
\]

\begin{thm}
The Banach space $\left(L_{\sigma},\,\left\Vert \cdot\right\Vert _{\sigma}\right)$
is reflexive iff the spectrum function $\sigma$ is unbounded, $\sigma\notin L^{\infty}$.\end{thm}
\begin{proof}
If $\sigma$ is bounded, then $L_{\sigma}=L^{1}$, the norms being
equivalent by Theorem~\ref{thm:Inclusions}~\ref{enu:LInfty}. But
$L^{1}$ is not a reflexive space and thus $\left(L_{\sigma},\,\left\Vert \cdot\right\Vert _{\sigma}\right)$
is not reflexive.

Secondly, assume  that $\sigma$ is unbounded, and let $\xi$ be a
continuous functional in the bi-dual, $\xi\in\left(L_{\sigma}^{*},\,\left\Vert \cdot\right\Vert _{\sigma}^{*}\right)^{*}$,
with norm $\left\Vert \xi\right\Vert <\infty$. Define the measure
$\nu\left(A\right):=\xi\left(\one_{A}\right)$. Let $A_{i}$ be a
sequence of mutually disjoint, measurable sets and set $A:=\bigcup_{i=1}^{\infty}A_{i}$.
Note, that 
\[
\left\Vert \one_{A}-\sum_{i=1}^{N}\one_{A_{i}}\right\Vert _{\sigma}^{*}=\left\Vert \sum_{i=N+1}^{\infty}\one_{A_{i}}\right\Vert _{\sigma}^{*}=\frac{1}{\frac{1}{p_{N}}\int_{1-p_{N}}^{1}\sigma\left(u\right)\mathrm{d}u}\le\frac{1}{\sigma\left(1-p_{N}\right)}\xrightarrow[N\to\infty]{}0
\]
(where $p_{N}:=\sum_{i=N+1}^{\infty}P\left(A_{i}\right)$) by \eqref{eq:12-1},
and because $\sigma$ is unbounded. From continuity of $\xi$ it follows
thus that $\nu$ is sigma additive. Further, if $P\left(A\right)=0$,
then $\left\Vert \one_{A}\right\Vert _{\sigma}^{*}=0$ and 
\[
\left|\nu\left(A\right)\right|\le\left\Vert \xi\right\Vert \cdot\left\Vert \one_{A}\right\Vert _{\sigma}^{*}=0,
\]
$\nu$ thus is absolutely continuous with respect to $P$. 

Let $Y\in L^{0}$ be the Radon\textendash{}Nikodým density, $\mathrm{d}\nu=Y\mathrm{d}P$,
for which 
\[
\xi\left(\one_{A}\right)=\nu\left(A\right)=\int_{A}Y\mathrm{d}P=\int Y\one_{A}\mathrm{d}P=\mathbb{E}\, Y\one_{A},
\]
and from linearity thus $\xi\left(\phi\right)=\mathbb{E}\, Y\phi$
and $\left|\mathbb{E}\, Y\phi\right|=\left|\xi\left(\phi\right)\right|\le\left\Vert \xi\right\Vert \cdot\left\Vert \phi\right\Vert _{\sigma}^{*}$
for a simple function $\phi$.

Finally let $U$ be coupled in a co-monotone way with $\left|Y\right|$
and let $\sigma_{n}$ be simple, nondecreasing step functions with
$0\le\sigma_{n}\le\sigma_{n+1}\to\sigma$ pointwise, then 

\begin{eqnarray*}
\left\Vert Y\right\Vert _{\sigma} & = & \rho_{\sigma}\left(\left|Y\right|\right)=\mathbb{E}\left|Y\right|\sigma\left(U\right)=\mathbb{E}\left|Y\right|\cdot\lim_{n\to\infty}\sigma_{n}\left(U\right)\\
 & \le & \liminf_{n\to\infty}\:\mathbb{E}\left|Y\right|\sigma_{n}\left(U\right)=\liminf_{n\to\infty}\:\mathbb{E}\, Y\cdot\sign\left(Y\right)\sigma_{n}\left(U\right)\\
 & \le & \liminf_{n\to\infty}\left\Vert \xi\right\Vert \cdot\left\Vert \sign\left(Y\right)\sigma_{n}\left(U\right)\right\Vert _{\sigma}^{*}=\left\Vert \xi\right\Vert \cdot\liminf_{n\to\infty}\left\Vert \sigma_{n}\left(U\right)\right\Vert _{\sigma}^{*}\\
 & \le & \left\Vert \xi\right\Vert \cdot\left\Vert \sigma\left(U\right)\right\Vert _{\sigma}^{*}=\left\Vert \xi\right\Vert <\infty
\end{eqnarray*}
by Fatou's Lemma, monotonicity, \eqref{eq:Monotonicity}  and \eqref{eq:25}.
This proves that $Y\in L_{\sigma}$, and $L_{\sigma}$ thus is reflexive.
\end{proof}
\medskip{}

The following statement compares $L_{\sigma}^{*}$ spaces with spaces
$L^{q}$, and it generalizes the relations \eqref{eq:18} and \eqref{eq:19}
for general $L^{q}$ spaces. It is the dual statement to Theorem~\ref{thm:Inclusions}.
\begin{thm}[Comparison with $L^{q}$]
For $\sigma\in L^{q}$ ($1\le q\le\infty$) it holds that 
\[
\left\Vert Z\right\Vert _{q}\le\left\Vert Z\right\Vert _{\sigma}^{*}\cdot\left\Vert \sigma\right\Vert _{q}
\]
whenever $Z\in L_{\sigma}^{*}$, and thus $L_{\sigma}^{*}\subset L^{q}$.

Moreover, 
\[
\frac{\left\Vert Z\right\Vert _{\infty}}{\left\Vert \sigma\right\Vert _{\infty}}\le\left\Vert Z\right\Vert _{\sigma}^{*}\le\left\Vert Z\right\Vert _{\infty}
\]
such that the norms $\left\Vert \cdot\right\Vert _{\infty}$ and $\left\Vert \cdot\right\Vert _{\sigma}^{*}$
are equivalent whenever $\sigma\in L^{\infty}$, and in this case
$L_{\sigma}^{*}=L^{\infty}$.\end{thm}
\begin{proof}
Employing $L^{p}-L^{q}$ duality and $L_{\sigma}-L_{\sigma}^{*}$
duality it holds that 
\[
\left\Vert Z\right\Vert _{q}=\sup_{Y\neq0}\frac{\mathbb{E}YZ}{\left\Vert Y\right\Vert _{p}}\le\sup_{Y\neq0}\frac{\left\Vert Y\right\Vert _{\sigma}\left\Vert Z\right\Vert _{\sigma}^{*}}{\left\Vert Y\right\Vert _{p}}\le\sup_{Y\neq0}\frac{\left\Vert \sigma\right\Vert _{q}\left\Vert Y\right\Vert _{p}\left\Vert Z\right\Vert _{\sigma}^{*}}{\left\Vert Y\right\Vert _{p}}=\left\Vert \sigma\right\Vert _{q}\cdot\left\Vert Z\right\Vert _{\sigma}^{*}
\]
by \eqref{eq:sigmaLp}.

The inequality, which is missing, is given by 
\[
\left\Vert Z\right\Vert _{\sigma}^{*}=\sup_{Y\neq0}\frac{\mathbb{E}YZ}{\left\Vert Y\right\Vert _{\sigma}}\le\sup_{Y\neq0}\frac{\left\Vert Y\right\Vert _{1}\left\Vert Z\right\Vert _{\infty}}{\left\Vert Y\right\Vert _{\sigma}}\le\sup_{Y\neq0}\frac{\left\Vert Y\right\Vert _{\sigma}\left\Vert Z\right\Vert _{\infty}}{\left\Vert Y\right\Vert _{\sigma}}=\left\Vert Z\right\Vert _{\infty},
\]
again by \eqref{eq:sigmaLp}.
\end{proof}

\section{\label{sec:LGeneral}The general natural domain space $L_{\mathscr{S}}$}

Kusuoka's theorem (Theorem~\ref{thm:Kusuoka}) and \eqref{eq:Spectral}
suggest to consider risk measures of the form 
\[
\rho_{\mathscr{S}}\left(\cdot\right):=\sup_{\sigma\in\mathscr{S}}\rho_{\sigma}\left(\cdot\right).
\]
To investigate this general type of risk measure we define the according
norm and space first.
\begin{defn}
The \emph{natural domain} of $\rho_{\mathscr{S}}$, where $\mathscr{S}$
is a collection of spectral functions, is 
\[
L_{\mathscr{S}}:=\left\{ Y\in L^{1}:\,\left\Vert Y\right\Vert _{\mathscr{S}}<\infty\right\} ,
\]
where 
\[
\left\Vert \cdot\right\Vert _{\mathscr{S}}:=\rho_{\mathscr{S}}\left(\left|\cdot\right|\right)=\sup_{\sigma\in\mathscr{S}}\rho_{\sigma}\left(\left|\cdot\right|\right)=\sup_{\sigma\in\mathscr{S}}\left\Vert \cdot\right\Vert _{\sigma}.
\]

\end{defn}
Obviously, $L_{\mathscr{S}}\subset\bigcap_{\sigma\in\mathscr{S}}L_{\sigma}$.
In view of Theorem~\ref{thm:Inclusions}~\ref{enu:LInfty} it is
obvious as well that 
\[
L^{\infty}\subset L_{\mathscr{S}}\subset L^{1},
\]
even more, it holds that $\left\Vert Y\right\Vert _{\mathscr{S}}\le\left\Vert Y\right\Vert _{\infty}$
whenever $Y\in L^{\infty}$, and $\left\Vert Y\right\Vert _{1}\le\left\Vert Y\right\Vert _{\mathscr{S}}$,
whenever $Y\in L_{\mathscr{S}}$. Further, if $\sup_{\sigma\in\mathscr{S}}\left\Vert \sigma\right\Vert _{q}<\infty$
is finite as well, then 
\[
\left\Vert Y\right\Vert _{\mathscr{S}}\le\sup_{\sigma\in\mathscr{S}}\left\Vert \sigma\right\Vert _{q}\cdot\left\Vert Y\right\Vert _{p}
\]
 by Theorem~\ref{thm:Inclusions},~\ref{enu:Lp}.
\begin{thm}
The pair $\left(L_{\mathscr{S}},\,\left\Vert \cdot\right\Vert _{\mathscr{S}}\right)$
is a Banach space.\end{thm}
\begin{proof}
First of all it is clear that $\left\Vert \cdot\right\Vert _{\mathscr{S}}$
is a norm on $L_{\mathscr{S}}$, as it separates points, is positively
homogeneous and satisfies the triangle inequality: these properties
are inherited from the spaces $\left(L_{\sigma},\,\left\Vert \cdot\right\Vert _{\sigma}\right)_{\sigma\in\mathscr{S}}$. 

It remains to be shown that $\left(L_{\mathscr{S}},\,\left\Vert \cdot\right\Vert _{\mathscr{S}}\right)$
is complete. So if $\left(Y_{k}\right)_{k}$ is a Cauchy sequence
in $L_{\mathscr{S}}$, then because of $\left\Vert \cdot\right\Vert _{\sigma}\le\left\Vert \cdot\right\Vert _{\mathscr{S}}$
it is a Cauchy sequence in any of the spaces $\left(L_{\sigma},\,\left\Vert \cdot\right\Vert _{\sigma}\right)$
and it has a limit $Y$ there. The limit is the same for all $L_{\sigma}$,
so $Y\in\bigcap_{\sigma\in\mathscr{S}}L_{\sigma}$. Following \eqref{eq:15}
it holds that 
\[
\left\Vert Y\right\Vert _{\mathscr{S}}=\sup_{\sigma\in\mathscr{S}}\left\Vert Y\right\Vert _{\sigma}\le\sup_{\sigma\in\mathscr{S}}\liminf_{k\to\infty}\left\Vert Y_{k}\right\Vert _{\sigma}\le\liminf_{k\to\infty}\sup_{\sigma\in\mathscr{S}}\left\Vert Y_{k}\right\Vert _{\sigma}=\liminf_{k\to\infty}\left\Vert Y_{k}\right\Vert _{\mathscr{S}}
\]
 by the max-min inequality. Now choose $k^{*}\in\mathbb{N}$ such
that $\left\Vert Y_{k}-Y_{k^{*}}\right\Vert _{\mathscr{S}}<1$ for
all $k>k^{*}$, which is possible because the sequence is Cauchy.
It follows that 
\[
\left\Vert Y\right\Vert _{\mathscr{S}}\le\liminf_{k\to\infty}\left\Vert Y_{k}\right\Vert _{\mathscr{S}}\le\left\Vert Y_{k^{*}}\right\Vert _{\mathscr{S}}+1<\infty,
\]
and hence $Y\in L_{\mathscr{S}}$, that is $L_{\mathscr{S}}$ is complete.\end{proof}
\begin{thm}
The risk measure $\rho_{\mathscr{S}}$ is finite valued on $L_{\mathscr{S}}$,
it is moreover continuous with respect to the norm $\left\Vert \cdot\right\Vert _{\mathscr{S}}$
with Lipschitz constant $1$.\end{thm}
\begin{proof}
The assertion follows from the more general Proposition~\ref{prop:Continuity}.
\end{proof}

\paragraph{Comparison of different $L_{\mathscr{S}}$ spaces. }

The norm of the identity 
\[
\id:\left(L_{\mathscr{S}_{1}},\left\Vert \cdot\right\Vert _{\mathscr{S}_{1}}\right)\rightarrow\left(L_{\mathscr{S}_{2}},\left\Vert \cdot\right\Vert _{\mathscr{S}_{2}}\right)
\]
is 
\[
\left\Vert \id\right\Vert =\sup_{\sigma_{2}\in\mathscr{S}_{2}}\:\inf_{\sigma_{1}\in\mathscr{S}_{1}}\:\sup_{0\le\alpha<1}\frac{\int_{\alpha}^{1}\sigma_{2}\left(u\right)\mathrm{d}u}{\int_{\alpha}^{1}\sigma_{1}\left(u\right)\mathrm{d}u},
\]
and $L_{\mathscr{S}_{1}}\subset L_{\mathscr{S}_{2}}$ iff $\left\Vert \id\right\Vert <\infty$.
This is immediate from \eqref{eq:9}, \eqref{eq:10} and 
\[
\left\Vert \id\right\Vert =\inf\left\{ c>0:\;\forall\sigma_{2}\in\mathscr{S}_{2}\;\exists\sigma_{1}\in\mathscr{S}_{1}:\,\int_{\alpha}^{1}\sigma_{2}\left(u\right)\mathrm{d}u\le c\cdot\int_{\alpha}^{1}\sigma_{1}\left(u\right)\mathrm{d}u\text{ for all }\alpha\in\left(0,1\right)\right\} .
\]

\subsection*{Examples}

We give finally two examples for which the norm $\left\Vert \cdot\right\Vert _{\mathscr{S}}$
induced by a set of spectral functions $\mathscr{S}$ coincides with
the norm $\left\Vert \cdot\right\Vert _{p}$ on $L^{p}$. Note, that
this is contrast to the space $L_{\sigma}$, as Theorem~\ref{thm:LInfty}
insures that $L_{\sigma}$ is strictly larger than $L^{p}$. 
\begin{example}[Higher order semideviation]
The $p-$semideviation risk measure for $0<\lambda\le1$ is 
\[
\rho\left(Y\right):=\mathbb{E}Y+\lambda\cdot\left\Vert \left(Y-\mathbb{E}Y\right)_{+}\right\Vert _{p}.
\]
Then $L_{\mathscr{S}}=L^{p}$, where $\mathscr{S}$ is an appropriate
spectrum to generate $\rho=\rho_{\mathscr{S}}$, and the norms $\left\Vert \cdot\right\Vert _{\mathscr{S}}$
and $\left\Vert \cdot\right\Vert _{p}$ are equivalent.\end{example}
\begin{proof}
The generating set $\mathscr{S}$ is provided in \cite{RuszczynskiShapiro2009}
and in \cite{ShapiroAlois}, the higher order semideviation risk measure
takes the alternative form 
\[
\rho\left(Y\right)=\rho_{\mathscr{S}}\left(Y\right)=\sup_{\sigma\in L^{q}}\left(1-\frac{\lambda}{\left\Vert \sigma\right\Vert _{q}}\right)\mathbb{E}Y+\frac{\lambda}{\left\Vert \sigma\right\Vert _{q}}\rho_{\sigma}\left(Y\right).
\]
It is evident that $\rho_{\mathscr{S}}\left(\left|Y\right|\right)\le\left(1-\frac{\lambda}{\left\Vert \sigma\right\Vert _{q}}\right)\left\Vert Y\right\Vert _{1}+\lambda\left\Vert Y\right\Vert _{p}\le\left(1+\lambda\right)\left\Vert Y\right\Vert _{p}$,
such that $\rho_{\mathscr{S}}$ is finite valued for $Y\in L^{p}$.
We claim that the natural domain is $L_{\mathscr{S}}=L^{p}$. For
this suppose that $Y\in L_{\mathscr{S}}\backslash L^{p}$, i.e. $\left\Vert Y\right\Vert _{1}<\infty$,
but $\left\Vert Y\right\Vert _{p}=\infty$. So it holds that 
\[
\rho_{\mathscr{S}}\left(Y\right)\ge\lambda\cdot\sup_{\sigma\in L^{q}}\frac{\rho_{\sigma}\left(Y\right)}{\left\Vert \sigma\right\Vert _{q}}=\lambda\cdot\sup_{Z\in L^{q}}\mathbb{E}Y\frac{Z}{\left\Vert Z\right\Vert _{q}}=\lambda\cdot\left\Vert Y\right\Vert _{p}=\infty
\]
by $L^{p}-L^{q}$ duality, hence $Y\notin L_{\mathscr{S}}$ and thus
$L_{\mathscr{S}}=L^{p}$. 

It follows by the open mapping theorem that the norms are equivalent.\end{proof}
\begin{example}
Theorem \ref{thm:LInfty} states that $L_{\sigma}\supsetneqq L^{\infty}$,
that is to say $L_{\sigma}$ is strictly larger than $L^{\infty}$.
This is not the case any more for the space $L_{\mathscr{S}}$: for
this consider just the risk measure 
\[
\rho\left(Y\right):=\sup_{\alpha<1}\:\AVaR_{\alpha}\left(Y\right)\qquad\left(=\esssup Y\right).
\]
Then $\rho\left(Y\right)<\infty$ if and only if $\esssup Y<\infty$,
that is $L_{\mathscr{S}}=L^{\infty}$.
\end{example}

\section{\label{sec:Summary}Summary}

In this paper we associate a norm with a risk measure in a natural
way. The risk measure is continuous with respect to the associated
norm. This point of view allows considering spectral risk measures
on its natural domain, which is a Banach space and as large as possible.
The space of natural domain is considerably larger than an accordant
$L^{p}$ space for spectral risk measures. 

As important representation theorems, as the Fenchel\textendash{}Moreau
theorem,  involve the dual space, we study the dual space as well.
Its norm can be described by a gauge functional, and the underlying
set is characterized by second order stochastic dominance constraints,
which measure the pace of growth of the random variable considered.

A consequence of the results of this paper is given by the fact that
finite valued risk measures cannot be defined on a space lager than
$L^{1}$ in a meaningful way.

\section{Acknowledgment}

The author is indebted to Prof. Alexander Shapiro (Georgia Tech) for
numerous discussions on this and other subjects, not only during the
work on this paper. In particular Theorem~\ref{thm:Shapiro} is attributed
to Prof. Shapiro.

\bibliographystyle{abbrv}
\bibliography{LiteraturAlois}

\begin{thebibliography}{10}

\bibitem{Acerbi2002a}
C.~Acerbi.
\newblock Spectral measures of risk: A coherent representation of subjective
  risk aversion.
\newblock {\em Journal of Banking \& Finance}, 26:1505--1518, 2002.

\bibitem{Acerbi2002}
C.~Acerbi and P.~Simonetti.
\newblock Portfolio optimization with spectral measures of risk.
\newblock {\em EconPapers}, 2002.

\bibitem{Aliprantis2006}
C.~D. Aliprantis and K.~C. Border.
\newblock {\em Infinite Dimensional Analysis}.
\newblock Springer, 2006.

\bibitem{Artzner1999}
P.~Artzner, F.~Delbaen, J.-M. Eber, and D.~Heath.
\newblock Coherent {M}easures of {R}isk.
\newblock {\em Mathematical Finance}, 9:203--228, 1999.

\bibitem{Artzner1997}
P.~Artzner, F.~Delbaen, and D.~Heath.
\newblock Thinking coherently.
\newblock {\em Risk}, 10:68--71, November 1997.

\bibitem{Cheridito2008}
P.~Cheridito and T.~Li.
\newblock Dual characterization of properties of risk measures on {O}rlicz
  hearts.
\newblock {\em Mathematics and Financial Economics}, 2(1):29--55, 2008.

\bibitem{Cheridito2009a}
P.~Cheridito and T.~Li.
\newblock Risk measures on {O}rlicz hearts.
\newblock {\em Mathematical Finance}, 19(2):189--214, 2009.

\bibitem{Cherny2006}
A.~S. Cherny.
\newblock Weighted {V@R} and its properties.
\newblock {\em Finance and Stochastics}, 10:367--393, 2006.

\bibitem{Dana2005}
R.-A. Dana.
\newblock A representation result for concave {S}chur concave functions.
\newblock {\em Mathematical Finance}, 15:613--634, 2005.

\bibitem{Denneberg1989}
D.~Denneberg.
\newblock Distorted probabilities and insurance premiums.
\newblock {\em Methods of Operations Research}, 63:21--42, 1990.

\bibitem{DentchevaRusz2004}
D.~Dentcheva and A.~Ruszczy\'nski.
\newblock Convexification of stochastic ordering.
\newblock {\em C. R. Acad. Bulgare Sci.}, 57(3):5--10, 2004.

\bibitem{Denuit}
M.~Denuit, J.~Dhaene, M.~Goovaerts, and R.~Kaas.
\newblock {\em Actuarial Theory for Dependent Risks: Measures, Orders and
  Models}.
\newblock Wiley, 2006.

\bibitem{Filipovic2012}
D.~Filipovi\'c and G.~Svindland.
\newblock The canonical model space for law-invariant convex risk measures is
  ${L}^1$.
\newblock {\em Mathematical Finance}, 22(3):585--589, 2012.

\bibitem{Follmer2004}
H.~F\"ollmer and A.~Schied.
\newblock {\em Stochastic Finance: An Introduction in Discrete Time}.
\newblock de Gruyter Studies in Mathematics 27. de Gruyter, 2004.

\bibitem{Hardy1988}
G.~H. Hardy, J.~E. Littlewood, and G.~Pólya.
\newblock {\em Inequalities}.
\newblock Cambridge University Press, 1988.

\bibitem{Hoeffding}
W.~Hoeffding.
\newblock Maßstabinvariante {K}orrelationstheorie.
\newblock {\em Schriften Math. Inst. Univ. Berlin}, 5:181--233, 1940.
\newblock In German.

\bibitem{Schachermayer2006}
E.~Jouini, W.~Schachermayer, and N.~Touzi.
\newblock Law invariant risk measures have the {F}atou property.
\newblock {\em Advances in Mathematical Economics}, 9:49--71, 2006.

\bibitem{Kusuoka}
S.~Kusuoka.
\newblock On law invariant coherent risk measures.
\newblock {\em Advances in Mathematical Economics}, 3:83--95, 2001.

\bibitem{PflugPichler}
G.~{\relax Ch}. Pflug and A.~Pichler.
\newblock Time consistency and temporal decomposition of positively homogeneous
  risk functionals.
\newblock Manuscript, 2013.

\bibitem{PflugRomisch2007}
G.~{\relax Ch}. Pflug and W.~R\"omisch.
\newblock {\em Modeling, Measuring and Managing Risk}.
\newblock World Scientific, River Edge, NJ, 2007.

\bibitem{Rockafellar1974}
R.~T. Rockafellar.
\newblock {\em Conjugate Duality and Optimization}, volume~16.
\newblock CBMS-NSF Regional Conference Series in Applied Mathematics. 16.
  Philadelphia, Pa.: SIAM, Society for Industrial and Applied Mathematics. VI,
  74 p., 1974.

\bibitem{Rudin1973}
W.~Rudin.
\newblock {\em Functional Analysis}.
\newblock McGraw-Hill, 1973.

\bibitem{Ruszczynski2006}
A.~Ruszczy\'nski and A.~Shapiro.
\newblock Optimization of convex risk functions.
\newblock {\em Mathematics of operations research}, 31:433--452, 2006.

\bibitem{RuszczynskiShapiro2009}
A.~Shapiro, D.~Dentcheva, and A.~Ruszczy\'nski.
\newblock {\em Lectures on {S}tochastic {P}rogramming}.
\newblock MQS-SIAM Series on Optimization 9, 2009.

\bibitem{ShapiroAlois}
A.~Shapiro and A.~Pichler.
\newblock Uniqueness of {K}usuoka representations.
\newblock Manuscript, 2013.

\bibitem{vdVaart}
A.~W. van~der Vaart.
\newblock {\em Asymptotic Statistics}.
\newblock Cambridge University Press, 1998.

\bibitem{Wang1997}
S.~S. Wang, V.~R. Young, and H.~H. Panjer.
\newblock Axiomatic characterization of insurance prices.
\newblock {\em Insurance: Mathematics and Economics}, 21:173--183, 1997.

\end{thebibliography}

\end{document}